\newcommand{\ExecuteBeforeCleveref}{%
    \usepackage[noend]{algpseudocode}%
}
\documentclass[a4paper,english,11pt]{texbase/shinyart} %twoside
\usepackage[utf8]{inputenc}
\usepackage{csquotes}
\usepackage{babel}
\usepackage[T1]{fontenc}
\usepackage{tikz}
\usepackage{pgfplots}
\pgfplotsset{compat=newest}
\pgfplotsset{plot coordinates/math parser=false}
\usetikzlibrary{plotmarks}
\newlength\figureheight
\newlength\figurewidth
\usepackage{grffile}
\usepackage{float}
\usepackage{enumitem}
\usepackage{comment}
\usepackage{undertilde}

\theoremstyle{definition}
\newtheorem{algorithm}{Algorithm}
\numberwithin{algorithm}{section}
\usepackage{bbm}
\usepackage{amssymb}
\usepackage{mdframed}



\mdfdefinestyle{examplebg}{
    backgroundcolor=black!7,%
    hidealllines=true,%
    innertopmargin=-.3em,%
    innerbottommargin=.7em,%
    innerleftmargin=.7em,%
    innerrightmargin=.7em,%
}

\surroundwithmdframed[style=examplebg]{example}
\surroundwithmdframed[style=examplebg]{algorithm}

\newcounter{Hequation}

\makeatletter
\g@addto@macro\equation{\stepcounter{Hequation}}
\makeatother

\usepackage[textsize=footnotesize,color=DarkOrange!40]{todonotes}

\newcommand{\term}{\emph}

\newcommand{\field}[1]{\mathbb{#1}}
\newcommand{\N}{\mathbb{N}}

\newcommand{\R}{\field{R}}
\newcommand{\extR}{\overline \R}

\newcommand{\norm}[1]{\|#1\|}

\newcommand{\abs}[1]{|#1|}

\newcommand{\inv}[1]{#1^{-1}}
\newcommand{\grad}{\nabla}
\newcommand{\freevar}{\,\boldsymbol\cdot\,}

\newcommand{\Union}\bigcup
\newcommand{\Isect}\bigcap
\newcommand{\union}\cup
\newcommand{\isect}\cap
\newcommand{\bigunion}\bigcup
\newcommand{\bigisect}\bigcap

\newcommand{\defeq}{:=}

\newcommand{\downto}{\searrow}

\newcommand{\subdiff}{\partial}
\newcommand{\rangeSymbol}{\mathcal{R}}
\newcommand{\range}[1]{\rangeSymbol(#1)}

\DeclareMathOperator*{\argmin}{arg\,min}

\DeclareMathOperator{\interior}{int}

\DeclareMathOperator{\bd}{bd}

\DeclareMathOperator{\Dom}{dom}

\DeclareMathOperator{\tr}{tr}

\makeatletter
\def \uminus@sym{\setbox0=\hbox{$\cup$}\rlap{\hbox 
        to\wd0{\hss\raise0.5ex\hbox{$\scriptscriptstyle{-}$}\hss}}\box0}
    \def \uminus    {\mathrel{\uminus@sym}}
\makeatother

\newcommand{\mathvar}[1]{\textup{#1}}

\renewcommand{\tilde}{\widetilde}

\newcommand{\iprod}[2]{\langle #1,#2\rangle}

\renewcommand{\L}{\mathcal{L}}

\newcommand{\BD}{\partial}
\newcommand{\TV}{\mathvar{TV}}

\makeatletter
\def \weaktostar@sym{\setbox0=\hbox{$\rightharpoonup$}\rlap{\hbox 
        to\wd0{\hss\raise1ex\hbox{$\scriptscriptstyle{*\,}$}\hss}}\box0}
    \def \weaktostar    {\mathrel{\weaktostar@sym}}
\makeatother

\newcommand{\jprod}{\circ}
\newcommand{\J}{\mathcal{J}}
\newcommand{\K}{\mathcal{K}}
\newcommand{\iK}{\interior \K}
\newcommand{\E}{\mathcal{E}}

\newcommand{\nullspace}[1]{\mathcal{N}(#1)}

\newcommand{\eigval}{\lambda}

\usepackage{mathtools}

\floatstyle{ruled}
\newfloat{Algorithm}{tbp!}{loa}
\crefname{Algorithm}{Algorithm}{Algorithms}

\newcommand{\setto}{\rightrightarrows}

\def\extR{\overline \R}
\def\linear{\mathcal{L}}

\newcommand{\linearLArrow}[1][]{\linear_{\triangleleft\ifx&#1&\else,\,#1\fi}}
\newcommand{\linearLArrowSpecial}[1][]{\linear^{\star}_{\triangleleft\ifx&#1&\else,\,#1\fi}}

\def\realopt#1{\widehat #1}
\def\this#1{#1^i}
\def\nexxt#1{#1^{i+1}}

\def\realoptu{{\realopt{u}}}

\def\realoptx{{\realopt{x}}}
\def\realoptz{{\realopt{z}}}
\def\realopty{{\realopt{y}}}
\def\realoptd{{\realopt{d}}}
\def\nextu{\nexxt{u}}
\def\nextx{\nexxt{x}}

\def\nexty{\nexxt{y}}

\def\nextd{\nexxt{d}}
\def\thisu{\this{u}}
\def\thisx{\this{x}}

\def\thisy{\this{y}}

\def\E{\mathbb{E}}

\def\Tau{T}
\def\TauTest{\Phi}
\def\tauTest{\phi}
\def\SigmaTest{\Psi}
\def\sigmaTest{\psi}

\def\BregFn{V}

\def\AltBregFn{V'}

\def\GammaLift#1{\Xi_{#1}}

\newcommand{\Test}{Z}
\newcommand{\Precond}{M}
\newcommand{\Step}{W}

\newcommand{\convex}{\mathcal{C}}

\DeclareFontFamily{U}{mathx}{\hyphenchar\font45}
\DeclareFontShape{U}{mathx}{m}{n}{<-> mathx10}{}
\DeclareSymbolFont{mathx}{U}{mathx}{m}{n}
\DeclareMathAccent{\widebar}{0}{mathx}{"73}

\newcommand{\Penalty}{\Delta}
\newcommand{\soc}{\mathrm{soc}}

\def\optyMU{y_\mu}
\def\optyMUtrans{{\tilde y}_\mu}
\def\optdMU{d_\mu}
\def\optzMU{z_\mu}
\def\esty{y'}
\def\estd{d'}
\def\estyTrans{{\tilde y}'}
\def\estdTrans{{\utilde d}'}
\def\realoptyTrans{\tilde{\realopty}}
\def\realoptdTrans{\utilde{\realoptd}}
\def\primid{e}
\def\PenaltyX{\delta}
\def\basepart#1{{\widebar #1}}
\def\basepartRealopty{\realopt{{\basepart y}}}
\def\basepartRealoptd{\realopt{{\basepart d}}}
\def\omegalowerbound#1{\omega_{*,#1}}

\def\bar{\widebar}

\title{Interior--proximal primal--dual methods}

\makeatletter
\AtBeginDocument{
\hypersetup{
    pdftitle = {Interior--proximal primal--dual methods},
    pdfauthor = {T.~Valkonen}
}
}
\makeatother

\begin{document}

\date{2017-06-21 (revised 2018-07-31)}
\author{
    Tuomo Valkonen\thanks{ModeMat, Escuela Politécnica Nacional, Quito, Ecuador; previously Department of Mathematical Sciences, University of Liverpool, United Kingdom. \email{tuomo.valkonen@iki.fi}}
    }

\maketitle

\begin{abstract}    
	We study preconditioned proximal point methods for a class of saddle point problems, where the preconditioner decouples the overall proximal point method into an alternating primal--dual method. This is akin to the Chambolle--Pock method or the ADMM.
    In our work, we replace the squared distance in the dual step by a barrier function on a symmetric cone, while using a standard (Euclidean) proximal step for the primal variable.
    We show that under non-degeneracy and simple linear constraints, such a hybrid primal--dual algorithm can achieve linear convergence on originally strongly convex problems involving the second-order cone in their saddle point form. 
    On general symmetric cones, we are only able to show an $O(1/N)$ rate.
    These results are based on estimates of strong convexity of the barrier function, extended with a penalty to the boundary of the symmetric cone.
    The main contributions of the paper are these theoretical results.

    \textbf{Due to arXiv's inability to handle biblatex properly, and refusal to accept PDFs, references are broken in this file. Please get the correctly typeset version from \url{http://tuomov.iki.fi/publications/}.}
\end{abstract}

%%%
\section{Introduction}
%%%

Interior point methods exhibit fast convergence on several non-smooth non-strongly-convex problems, including linear problems with symmetric cone constraints \cite{as-2003,faybusovich1997euclidean,monteiro1998polynomial,nesterov-todd}. 
The methods have had less success on large-scale problems with more complex structure.
In particular, problems in image processing, inverse problems, and data science, can often be written in the form
\begin{equation*}
	\tag{P}
	\label{eq:problem}
	\min_x G(x) + F(Kx)
\end{equation*}
for convex, proper, lower semicontinuous $G$ and $F$, and a bounded linear operator $K$.
Often, with $G$ and $F$ involving norms and linear operators, \eqref{eq:problem} can be converted into linear optimisation on symmetric cones.
This is even automated by the disciplined convex programming approach of CVX \cite{cvx,gb08}.
Nonetheless, the need to solve a very large scale and difficult Newton system on each step of the interior point method makes this approach seldom practical for real-world problems.
Therefore, first-order splitting methods such as forward--backward splitting, ADMM (alternating directions method of multipliers) and their variants \cite{beck2009fista,loris2011generalization,gabay,chambolle2010first} dominate these application areas. In our present work, we are curious \emph{whether these two approaches---interior point and splitting methods---can be combined into an effective algorithm?}

The saddle point form of \eqref{eq:problem} is
\begin{equation*}
	\tag{S}
	\label{eq:saddle}
	\min_x ~\max_y~ G(x) + \iprod{Kx}{y} - F^*(y).
\end{equation*}
A popular algorithm for solving this problem is the primal--dual method of Chambolle and Pock \cite{chambolle2010first}. As discovered in \cite{he2012convergence}, the method can most concisely be written as a \term{preconditioned proximal point method}, solving on each iteration for $\nextu=(\nextx,\nexty)$ the variational inclusion 
\begin{equation*}
	\tag{PP$_0$}
	\label{eq:pp0}
	0 \in H(\nextu) + \Precond_{i+1}(\nextu-\thisu),
\end{equation*}
where the monotone operator
\begin{equation*}
    %\label{eq:h}
    H(u) \defeq
        \begin{pmatrix}
            \subdiff G(x) + K^* y \\
            \subdiff F^*(y) -K x
        \end{pmatrix}
\end{equation*}
encodes the optimality condition $0 \in H(\realoptu)$ for \eqref{eq:saddle}.
For the standard proximal point method \cite{rockafellar1976monotone}, one would take $M­_{i+1}=I$ the identity. With this choice, the system \eqref{eq:pp0} is generally difficult to solve. 
In the Chambolle--Pock method the \term{preconditioning} or step length operator is given for suitably chosen step length parameters $\tau_i,\sigma_{i+1},\theta_i>0$ by
\[
	\Precond_{i+1} \defeq \begin{pmatrix} \inv\tau_i I  & -K^* \\ -\theta_i K & \inv\sigma_{i+1} I \end{pmatrix}.
\]
This choice of $\Precond_{i+1}$ decouples the primal $x$ and dual $y$ updates, making the solution of \eqref{eq:pp0} feasible in a wide range of problems.
If $G$ is strongly convex, the step length parameters $\tau_i,\sigma_{i+1},\theta_i$ can be chosen to yield $O(1/N^2)$ convergence rates of an ergodic duality gap and the squared distance $\norm{\thisx-\realoptx}^2$. If both $G$ and $F^*$ are strongly convex, then the method converges linearly.
Without any strong convexity, only the ergodic duality gap converges at the rate $O(1/N)$, and the iterates weakly \cite{tuomov-proxtest}.

In our earlier work \cite{tuomov-cpaccel,tuomov-proxtest,tuomov-blockcp}, we have modified $\Precond_{i+1}$ as well as the condition \eqref{eq:pp0} to still allow a level of mixed-rate acceleration when $G$ is strongly convex only on sub-spaces or sub-blocks of the variable $x=(x_1,\ldots,x_m)$, and derived a corresponding doubly-stochastic block-coordinate descent method. As an extension of that work, our specific question now is:
\begin{equation}
    \label{eq:intro-constr}
	\text{If } F^* \text{ encodes the constraint } Ay=b \text{ and } y \in \K
\end{equation}
for a symmetric cone $\K$, \emph{can we replace $\Precond_{i+1}$ in \eqref{eq:pp0} by a non-linear interior point preconditioner that yields tractable sub-problems and a fast, convergent algorithm?} 

Our approach is motivated, firstly, by the fact that \eqref{eq:intro-constr} frequently occurs in applications, in particular with $\K$ the \term{second-order cone} of elements $y=(y_0, \bar y) \in \R^{1+n}$ with $y_0 \ge \norm{\bar y}$ and $Ay=y_0$. This can be used to model $F^*$ that could otherwise be written as the constraint $F^*(y)=\delta_{B(0, b_0)})(\bar y)$.
Secondly, why we specifically want to try the interior point approach is that the standard and generic quadratic proximal term or preconditioner is not in any specific way adapted to the structure of the ball constraint or the cone $\K$: it is a penalty, but not a barrier approach. The logarithmic barrier, on the other hand, is exactly tuned to the structure of the problem. This suggests that it \emph{might} be able to yield better performance.

Generalised proximal point methods motivated by interior point methods have been considered before in \cite{yu2013convergence,chen2010entropy,lopez2017construction,kaplan2007interior,trandinh2014inexact}.
For iterates, which are generally shown to convergence, no convergence rates appear to be known. For function values, rates have been derived in \cite{lopez2017construction,trandinh2014inexact}.
This is in contrast to the superlinear convergence of a gap functional in conventional interior point methods for linear programming on symmetric cones \cite{as-2003,faybusovich1997euclidean,monteiro1998polynomial,nesterov-todd}.
The approach in the aforementioned works combining proximal point and interior point methods 
has essentially been to replace the squared distance in the proximal point method $\nextx \defeq \argmin_{x \in \K} G(x) +\frac{1}{2\tau}\norm{x-\thisx}^2$ for $\min_{x \in \K} G(x)$ by a suitable Bregman distance supported on $\iK \times \iK$, typically $D(x, x') \defeq \tr(x \jprod \ln x - x \jprod \ln x'+x'-x)$.
In \cref{sec:pedi} of the present work, we will instead replace the squared distance in the proximal point step for the dual variable $y$ by a more conventional barrier-based preconditioner $-\grad \log \det(y)$.
With this, we are able to obtain convergence rates \emph{for the iterates} of the method:
in general symmetric cones this is only $O(1/N)$ for the squared distance $\norm{x^N-\realoptx}^2$ between the primal iterate and the primal solution. I the second-order cone under non-degeneracy and $A=\iprod{a}{\freevar}$ for $a \in \iK$, this convergence becomes linear.
We demonstrate these theoretical results by numerical experiments in \cref{sec:numeric}.

The overall idea, how the theory works, is that the barrier-based preconditioner is strongly monotone on bounded subsets of $\iK$, and ``compatible'' with $\subdiff F^*$ on $\BD\K$ in such a way that these strong monotonicity estimates can, with some penalty term, be extended up to the boundary. This introduces some of the strong monotonicity that $\subdiff F^*$ itself is missing.

Since the performance of the overall algorithm we derive does not improve upon existing methods, our main contributions are these theoretical results on symmetrical cones.
An interesting question for future research is, whether the results for general cones can be improved, or whether the second-order cone is special?
Nevertheless, our present theoretical results make progress towards closing the gap between direct methods for \eqref{eq:problem}, and primal--dual methods for \eqref{eq:saddle}: among others, forward--backward splitting for \eqref{eq:problem} is known to obtain linear convergence with strong convexity assumptions on $G$ alone \cite{chen1997convergence}, but primal--dual methods generally still require the strong convexity of $F^*$ as well. For ADMM additional local estimates exist under quadratic \cite{boley2013local,han2013local} or polyhedrality assumptions \cite{hong2017linear}. On the other hand, it has been recently established that forward--backward splitting converges at least locally linearly even under less restrictive assumptions than the strong convexity of $G$ \cite{liang2014local,bredies2008linear}. 

Our convergence results depend on the convergence theory for non-linearly preconditioned proximal point methods from \cite{tuomov-proxtest}. We quote the relevant aspects in \cref{sec:general-estimates}. To use this theory, we need to compute estimates on the strong convexity of the barrier, with a penalty up to the boundary. This is the content of the latter half of \cref{sec:concepts}, after introduction of the basic Jordan-algebraic machinery for interior point methods on symmetric cones.

%\todo[inline]{This is very interesting, READ: \cite{lopez2017construction}. Also \cite{kaplan2007interior} \ldots mention ``other works look at using different interior-based Bregman distances in place of the squared Euclidean distance''\ldots these provide interesting possibilities \ldots however no rates so far achieved even in the primal setting}

%%%
\section{Notation, concepts, and results on symmetric cones}
\label{sec:concepts}
%%%

We write $\linear(X; Y)$ for space of bounded linear operators between Hilbert spaces $X$ and $Y$. For any $A \in \linear(X; Y)$ we write $\nullspace{A}$ for the null-space, and $\range{A}$ for the range.
%For $T \in \linear(X; X)$, the notation $T \ge S$ means that $T-S$ is positive semidefinite.
% In this case, we also denote
% \[
%     [0, T] \defeq \{\lambda T \mid \lambda \in [0, 1]\}.
% \]
%For $M \in \L(X; X)$, which can possibly not be self-adjoint, we employ the notation
Also for possibly non-self-adjoint $T \in \L(X; X)$, we introduce the inner product and norm-like notations
\begin{equation}
    \label{eq:iprod-def}
    \iprod{x}{z}_T \defeq \iprod{Tx}{z},
    \quad
    \text{and}
    \quad
    \norm{x}_T \defeq \sqrt{\iprod{x}{x}_T},
    \quad (x, z \in X).
\end{equation}
%We denote $T \simeq T'$ if $\norm{\freevar}_T = \norm{\freevar}_{T'}$.

With $\extR \defeq [-\infty,\infty]$, we write $\convex(X)$ for the space of convex, proper, lower semicontinuous functions from $X$ to $\extR$.
With $K \in \linear(X; Y)$, $G \in \convex(X)$ and $F^* \in \convex(Y)$ on Hilbert spaces $X$ and $Y$, we then wish to solve the minimax problem \eqref{eq:saddle}
assuming the existence of a solution $\realoptu=(\realoptx, \realopty)$ satisfying the optimality conditions $0 \in H(\realopty)$, in other words
\begin{equation}
    \label{eq:oc}
    \tag{OC}
    -K^* \realopty \in \subdiff G(\realoptx),
    \quad\text{and}\quad
    K \realoptx \in \subdiff F^*(\realopty).
\end{equation}

For a function $G$, as above, $\subdiff G$ stands the convex subdifferential \cite{rockafellar-convex-analysis}. For a set $C$,  $\subdiff C$ is the boundary.
We denote by $N_C(x)=\subdiff \delta_C(x)$ the normal cone to any convex set $C$ at $x \in C$, where $\delta_C$ is the indicator function of the set $C$ in the sense of convex analysis.

In \cref{sec:pedi}, we concentrate on $F^*$ of the general form \eqref{eq:lifted} in the next example.

\begin{example}[From ball constraints to second-order cones]
	\label{ex:ball-constraints}
    Very often in \eqref{eq:problem}, we have $F(z)=\sum_{i=1}^n \alpha_i \norm{z_i}_2$, where the norm is the Euclidean norm on $\R^m$ and $z=(z_1,\ldots,z_n) \in \R^{mn}$. Then $F^*(\basepart y)=\delta_{B(0, \alpha_i)}(\basepart y_i)$ for $\basepart y=(\basepart y_1,\ldots,\basepart y_n) \in \R^{mn}$.
    We may lift each $\basepart y_i$ into $\R^{1+m}$ as $y_i=(y_{i,0},\basepart y_i)$, and replace $F^*$ by 
    \begin{equation}
        \label{eq:lifted}
        \hat F^*(y) \defeq \sum_{i=1}^n \delta_{C_i}(y_i),
        \quad\text{where}\quad
        C_i \defeq \{ y_i \in \K \mid Ay=b\},
    \end{equation}
    where, the linear constraint is defined by $Ay \defeq (y_{1,0}, \ldots, y_{n,0})$ and $b \defeq (\alpha_1, \ldots, \alpha_n)$.
    The cone constraint is given by $\K = \K_\soc^n$ for the \term{second-order cone}
    \[
        \K_\soc \defeq \{ y=(y^0, \basepart y) \in \R^{1+m} \mid y_0 \ge \norm{\basepart y}\}.
    \]
\end{example}

In the following, we look at the Jordan-algebraic approach to analysis on the second-order cone and other \term{symmetric cones}.
%, which also include the positive definite cone.

\subsection{Euclidean Jordan algebras}
\label{sec:eucj}

We now introduce the minimum amount of the theory of Jordan algebras necessary for our work. For further details, we refer to \cite{faraut-koranyi-symmetric,koecher-notes}.

Technically, a real \term{Jordan algebra} $\J$ is a real (additive) vector space together with a bilinear and commutative multiplication operator $\circ: \J \times \J \to \J$ that satisfies the associativity condition $x \circ (x^2 \circ y) = x^2 \circ (x \circ y)$. Here we define $x^2 \defeq x \circ x$.
The Jordan algebra $\J$ is \emph{Euclidean} (or \emph{formally real}) if $x^2 + y^2 = 0$ implies $x=y=0$. We always assume that our Jordan algebras are Euclidean.

We will not directly need the last two technical definitions, but do rely on the very important consequence that $\J$ has a multiplicative unit element $e$: $x \jprod e = x$ for all $x \in \J$.
An element $x$ of $\J$ is then called invertible, if there exists an element $\inv x$, such that $x \jprod \inv x = \inv x \jprod x = e$.

\begin{example}[The Jordan algebra of symmetric matrices]
    \label{ex:matrix}
    To understand these and the following properties, it is helpful to think of the set of symmetric $m \times m$ matrices. They form a Jordan algebra endowed with the product $A \jprod B \defeq \frac{1}{2}(AB+BA)$.
    The inverse is the usual matrix inverse, as is the multiplicative identity.
    So are the properties discussed next.
\end{example}

An element $c$ in a Jordan algebra $\J$ is an \term{idempotent} if $c \jprod c = c$.
It is \term{primitive}, if it is not the sum of other idempotents.
A \term{Jordan frame} is a set of primitive idempotents $\{\primid_i\}_{i=1}^r$ such that $\primid_i \jprod \primid_j = 0$ for $i \ne j$, and $\sum_{j=1}^r \primid_j = e$. The number $r$ is the \term{rank} of $\J$.
For each $x \in \J$, there indeed exist unique real 
numbers $\{\eigval_i\}_{i=1}^r$, and a Jordan frame $\{\primid_i\}_{i=1}^r$, satisfying $x = \sum_{j=1}^r \eigval_i \primid_i$.
The numbers $\lambda_i(x)=\lambda_i$ are called the \term{eigenvalues} of $x$.
If all the eigenvalues are positive, we write $x > 0$ and call $x$ \term{positive definite}.
Likewise we write $x \ge 0$ if the eigenvalues are non-negative, and call $x$ \term{positive semi-definite}.
%The number of non-zero eigenvalues is the \emph{rank} of $x$.
With the eigenvalues, we can define
\begin{enumerate}[label=(\roman*),nosep]
	\item\label{item:eucj-power}
	Powers $x^\alpha \defeq \sum_{j=1}^r \eigval_i^\alpha \primid_i$ when meaningful,
	\item The determinant $\det x \defeq \prod_j \eigval_j$, and
	\item The trace $\tr x \defeq \sum_{j=1}^r \eigval_j$.
	\item The inner product $\iprod{x}{y} \defeq \tr (x \jprod y)$, and the
	\item Frobenius norm $\norm{x} \defeq \norm{x}_F \defeq \sqrt{\sum_{j=1}^r \eigval_j^2} = \sqrt{\iprod{x}{x}}$.
\end{enumerate}
The inner product is positive-definite and associative, satisfying $\iprod{x \jprod y}{z}=\iprod{y}{x \jprod z}$.  We also frequently write
\[
	\eigval_{\max}(x) \defeq \max_{i=1,\ldots,r} \eigval_i(x)
	\quad\text{and}\quad
	\eigval_{\min}(x) \defeq \min_{i=1,\ldots,r} \eigval_i(x).
\]

For conciseness, we define for $x \in \J$ the operator $L(x)$ by $L(x)y \defeq x \jprod y$.
The \term{quadratic presentation} of $x$---this is one of the most crucial concepts for us, as we will soon see when covering symmetric cones---is then defined as $Q_x \defeq 2 L(x)^2 - L(x^2)$. The invertibility of $x$ is equivalent to the invertibility of $Q_x$.
Other important properties include \cite{faraut-koranyi-symmetric,as-2003}
\begin{enumerate}[label=(\roman*),nosep,resume]
	\item $Q_x^\alpha=Q_{x^\alpha}$ for $\alpha \in \R$, 
	\item\label{item:eucj-fundamental} $Q_{Q_x y}=Q_x Q_y Q_x$ (the \term{fundamental formula} of quadratic presentations),  
	\item $Q_x \inv x = x$, 
	\item $Q_x e = x^2$, and
	\item\label{item:det-Qxy} $\det(Q_x y)=\det(x^2)y=\det(x)^2 y$.
\end{enumerate}
Moreover, $Q_x$ is self-adjoint with respect to the inner product defined above, and the eigenvalues are products $\eigval_i(x)\eigval_j(x)$ \cite{koecher-notes,faraut-koranyi-symmetric}, so that
\begin{equation}
	\label{eq:eigval-q-bound}
	\eigval_{\min}^2(x) \norm{y}^2 \le \iprod{Q_x y}{y} \le \eigval_{\max}(x)^2 \norm{y}^2
	\quad\text{for all}\quad y 
	\quad\text{when}\quad x \ge 0.
\end{equation}

\begin{example}[The Euclidean Jordan algebra of quadratic forms]
	\label{ex:qforms}
    Let $\E_{1+m}$ denote the space of vectors $x=(x_0, \basepart x) \in \R^{1+m}$ with $x_0$ scalar.
    Setting
    \[
        x \jprod y = (x^T y, x_0 \basepart y + y_0 \basepart x),
    \]
    we make $(\E_{1+m}, \jprod)$ into a Euclidean Jordan algebra.
    The identity element is $e=(1, 0)$, rank $r=2$, and the inner product is
    \begin{equation}
        \label{eq:qforms-innerproduct}
        \iprod{x}{y} = 2 x^T y.
    \end{equation}
    %The operator $L(x)$ is given by the matrix representation
    %\[
    %    L(x) = \Arw(x) \defeq \begin{pmatrix}
    %                             x_0 & \basepart x^T \\
    %                             \basepart x & x_0 I
    %                           \end{pmatrix}
    %\]
    %with $I$ the identity matrix.
    Defining the diagonal mirroring operator $R \defeq \left(\begin{smallmatrix} 1 & 0 \\ 0 & -I \end{smallmatrix}\right)$, we find that $\det x = x^T R x = x_0^2 - \norm{\basepart x}^2$, and $\inv x = R x / \det x$ when $\det x \ne 0$.
\end{example}

%%%
\subsection{Symmetric cones}
\label{sec:symmc}
%%%

The \term{cone of squares} of a Euclidean Jordan algebra $\J$ is defined as
\[
	\K \defeq \{ x^2 \mid x \in \J\}.
\]
The cones generated this way are precisely the so-called symmetric cones \cite{faraut-koranyi-symmetric} $\K^*=-\K$, or the self-scaled cones of \cite{nesterov-todd}. 
Their important properties include \cite{faraut-koranyi-symmetric,koecher-notes}:
\begin{enumerate}[label=(\roman*),nosep]
    \item\label{item:symmc:ik-l}
    	$\iK = \{ x \in \J \mid x \text{ is positive-definite}\}
               = \{ x \in \J \mid L(x) \text{ pos. def.}\}$.
    \item %$\iprod{x}{y} \ge 0$ for $x, y \in \K$, strictly if 
          %$x \in \iK, y \ne 0$.
          \label{property:sc-iprod-a}
          $\iprod{x}{y} \ge 0$ for all $y \in \K$ iff $x \in \K$, and
    \item \label{property:sc-iprod-b}
          $\iprod{x}{y} > 0$ for all $y \in \K \setminus \{0\}$ iff $x \in \iK$.
    %\item If $\iprod{x}{y} > 0$ for all $y \in \K \setminus \{0\}$, then
    %      $x \in \iK$.
    \item\label{property:sc-iprod-qonto} $Q_x$ for $x \in \iK$ maps $\K$ onto itself.
    \item For $x, y \in \iK$, there exists unique $a \in \iK$, such that $x=Q_a y$.
    \item\label{property:sc-iprod-zero} For any $x,y \in \K$, $\iprod{x}{y}=0$ iff $x \jprod y = 0$ \cite{faybusovich-1997}. 
\end{enumerate}
For application to interior point methods, and in particular for our work, the following properties are particularly important:
\begin{enumerate}[label=(\roman*),resume,nosep]
    \item The barrier function $B(x) \defeq -\log (\det x)$ tends to infinity as
          $x$ goes to $\bd \K$.
    \item\label{item:barrier} $\grad B(x) = -\inv x$ and $\grad^2 B(x) = \inv Q_x$ 
          (differentiated wrt.~the norm in $\J$).
%    \item $\norm{y}_x \defeq \norm{Q_x^{-1/2} y}_F$ defines a local norm 
%          around $x \in \iK$, such that $\norm{y-x}_x = \norm{Q_x^{-1/2}y-e}_F \le 1$ 
%          implies $y \in \K$. (This follows by considering the eigenvalue 
%          definition of $\norm{\cdot}_F$, and the onto-property of $Q_x$;
%          cf. also \cite{nesterov-todd}.)
	\item\label{item:symmc-normal} The normal cone $N_\K(x) = - \{y \in \K \mid \iprod{y}{x}=0\}$ for $x \in \K$ \cite[Lemma 3.1]{tuomov-thesis}.
\end{enumerate}

\begin{example}[The cone of symmetric positive definite matrices]
    In the Jordan algebra of symmetric matrices from \cref{ex:matrix}, the cone of squares is the set of positive semi-definite symmetric matrices.
    %, and its interest the set of positive definite symmetric matrces.
\end{example}

\begin{example}[The second order cone]
    The cone of squares of the Jordan algebra $\E_{1+m}$  of quadratic forms is the second order cone that we have already seen in \cref{ex:ball-constraints}, 
    \[
    	\K = \K_\soc \defeq \{ x \in \E_{1+m} \mid x_0 \ge \norm{\basepart x}\}.
    \]
    If $0 \ne x \in \bd \K$, we have $x^2 = 2 x_0 x$. Rescaled, we get a primitive idempotent $c=x/\sqrt{2x_0}$. The only primitive idempotent orthogonal to $c$ is $c'=Rx/\sqrt{2x_0}$. Therefore, the normal cone $N_\K(x) = \{-\alpha Rx \mid \alpha \ge 0\}$.
    %, a set  approximated by $\{-\alpha \inv y \mid \alpha \ge 0\}$ for invertible $y$ close to $x$.

    One has to be careful with the fact that the expressions for the barrier gradient and Hessian in \ref{item:barrier} are based on the inner product \eqref{eq:qforms-innerproduct} in $\E_{1+m}$. This is scaled by the factor $r=2$ with respect to the standard inner product on $\R^{1+m}$.
    %If ng with respect to  the standard inner product in $\R^{1+m}$, a correction of factor $2$ is required.
\end{example}

%%%
\subsection{Linear optimisation on symmetric cones}
\label{sec:sclp}
%%%

Let $A \in \L(\J; \R^k)$ for an arbitrary Euclidean Jordan algebra $\J$ with the corresponding cone of squares $\K$.
We will frequently make use of solutions $(\optyMU, \optdMU, \optzMU) \in \iK \times \iK \times \R^k$ to the system
\begin{equation*}
	\label{eq:ip-system-reg}
	\tag{SCLP$_\mu$}
	Ay = b, \quad A^*z + c = d, \quad y \circ d = \mu e, \quad y, d \in \iK.
\end{equation*}
These are meant to approximate solutions $(\realopty, \realoptd, \realoptz) \in \K \times \K \times \R^k$ to the system
\begin{equation*}
	\label{eq:ip-system}
	\tag{SCLP}
	Ay = b, \quad A^*z + c = d, \quad y \circ d = 0, \quad y,d \in \K.
\end{equation*}
The system \eqref{eq:ip-system} arises from primal--dual optimality conditions for linear optimisation on symmetric cones, specifically the problem
\begin{equation*}
	%\tag{SCLP}
    \min_{y \in \K,\, Ay=b}~\iprod{c}{y}.
\end{equation*}
The system \eqref{eq:ip-system-reg} arises from the introduction of the barrier in the problem
\begin{equation}
	\label{eq:barrier-problem}
	%\tag{SCLP$_\mu$}
    \min_{y \in \K,\, Ay=b}~\iprod{c}{y} - \mu \log \det(y).
\end{equation}

The set of solutions to \eqref{eq:ip-system-reg} for varying $\mu>0$ is called the \term{central path}. From \cite[Theorem 2.2]{faybusovich1997euclidean} we know that if there exists a primal--dual \term{interior feasible point}, i.e., some $(y^*, d^*, z^*) \in \iK \times \iK \times \R^k$ such that $Ay^*=b$ and $A^*z^*+c=d^*$, then there exists a solution $(\optyMU, \optdMU, \optzMU)$ to \eqref{eq:ip-system-reg} for every $\mu>0$. In particular, if there exists a solution for some $\mu>0$, there exist a solution for all $\mu>0$.
In fact, we have the following:

\begin{lemma}
	\label{lemma:sclp-existence}
    Suppose the primal feasible set $C \defeq \{ y \in \K \mid Ay=b\}$ is bounded, and that there exists a primal interior feasible point $y^* \in \iK \isect C$. Then there exists a solution $(y_\mu, d_\mu, z_\mu) \in \iK \times \iK \times \R^k$ to \eqref{eq:ip-system-reg} for all $\mu>0$.
\end{lemma}

\begin{proof}
	The article \cite{faybusovich1997euclidean} considers a more general class of linear monotone complementarity problems (LMCPs) than our our SCLPs (symmetric cone linear programs). For the special case of SCLPs, our assumption on the existence of $y^*$ implies that the feasible set in \eqref{eq:barrier-problem} non-empty and closed. Since the objective function is level-bounded, proper, and lower semicontinuous, the problem \eqref{eq:barrier-problem} has a solution $y$. This $y$ has to satisfy \eqref{eq:ip-system-reg} for some $d$ and $z$. Now \cite[Theorem 2.2]{faybusovich1997euclidean} applies.
\end{proof}

Practical methods \cite{nesterov-todd,as-2003} for solving \eqref{eq:ip-system} by closely following the central path are based on scaling the iterates $(y^i, d^i)$ by $Q_p$ for a suitable $p \in \iK$. We will need this scaling for different purposes, and therefore recall the following basic properties.

\begin{lemma}
    \label{lemma:scaling}
     Let $p \in \iK$, and $y, d \in \K$. Define $\tilde y \defeq Q_p^{1/2} y$, and $\utilde d \defeq Q_p^{-1/2} d$. Then
    \begin{enumerate}[label=(\roman*),nosep]
        \item\label{item:scaling-zero} $y \jprod d = 0$ if and only if $\tilde y \jprod \utilde d=0$.
        \item\label{item:scaling-mu} If $y,d \in \iK$ and $\mu>0$, then $y \jprod d = \mu e$ if and only if $\tilde y \jprod \utilde d = \mu e$.
        \item\label{item:scaling-lin} \eqref{eq:ip-system} (resp.~\eqref{eq:ip-system-reg}) is satisfied for $y$ and $d$ if and only if it is satisfied for $\tilde y$ and $\utilde d$ with $A$ and $c$ replaced by $\tilde A \defeq AQ_p^{-1/2}$ and $\utilde c \defeq Q_p^{-1/2} c$.
    \end{enumerate}
\end{lemma}

\begin{proof}
    The claim \ref{item:scaling-zero} is a consequence of the properties \cref{sec:symmc}\ref{property:sc-iprod-qonto} and \ref{property:sc-iprod-zero}.
    The claim \ref{item:scaling-lin} is the content of \cite[Lemma 28]{as-2003}.
    Finally, to establish \ref{item:scaling-lin}, the remaining linear equations in \eqref{eq:ip-system} and \eqref{eq:ip-system-reg} are obvious. 
\end{proof}

As a last preparatory step, before starting to derive new results, we say that solutions $y,d \in \K$ to \eqref{eq:ip-system} are \term{strictly complementary} if $y \jprod d=0$ and $y + d \in \iK$. We say that $y$ is \emph{primal non-degenerate} if % \cite{faybusovich-1997}
\begin{equation}
    \label{eq:primal-nondeg}
    v =A^*z \text{ and }  y \jprod v=0 \implies v=0.
\end{equation}
Likewise $d$ is \emph{dual non-degenerate} if 
\begin{equation}
    \label{eq:dual-nondeg}
    Av=0 \text{ and }  d \jprod v=0 \implies v=0.
\end{equation}

%%%
\subsection{Convergence rate of the central path}
\label{sec:approx}
%%%

We now study convergence rates for the central path, which we will need to develop approximate strong monotonicity estimates.
Some existing work can be found in \cite{wright2002properties}, but overall the results in the literature are limited; more work can be found on the properties and mere existence of limits of the central path \cite{ramirez2016central,dacruzneto2008central,monteiro1998existence,burachik2008properties,halicka2005limiting}.
After all, in typical interior point methods, one is not interested in solving \eqref{eq:ip-system-reg} exactly; rather, one is interested in staying close to the central path while decreasing $\mu$ fast.
So here we provide the result necessary for our work.

\begin{lemma}
    \label{lemma:central-path-approximation}
    Let $\realopty, \realoptd \in \K$ and $\realoptz \in \R^k$ solve \eqref{eq:ip-system}.
    Also let $\optyMU, \optdMU \in \iK$ and $\optzMU \in \R^k$ solve \eqref{eq:ip-system-reg} for some $\mu>0$.
    If $\realopty$ and $\realoptd$ are strictly complementary, and both primal and dual non-degenerate, then
    \begin{equation}
        \label{eq:central-path-approximation}
        \norm{\optyMU- \realopty}
        \le \frac{2\mu \sqrt{r}}{\eigval_{\min}(M_{\realopty,\realoptd})},
    \end{equation}
    where $\eigval_{\min}(M_{y,d}) > 0$ is the minimal eigenvalue of the the linear operator $M_{y,d} \in \L(\J; \J)$ defined at $y, d \in \J$ for $\eta \in \nullspace{A}$ and $\xi \in \range{A^*}$ by
    \[
        M_{y,d}(\xi+\eta) \defeq L(y)\xi + L(d)\eta.
    \]
\end{lemma}

\begin{proof}
    Observe that $(\optyMU, \optdMU, \optzMU)$ solves \eqref{eq:ip-system-reg} if and only if $\optyMU=\realopty+\Delta y$ and $\optdMU=\realoptd+\Delta d$ with
    \[
        \Delta y \in \nullspace{A},\quad
        \Delta d \in \range{A^*},\quad\text{and}\quad
        M_{\realopty,\realoptd}(\Delta y+\Delta d) = \mu e - \Delta y \jprod \Delta d.
    \]
    Here we have used the fact that $\realopty \jprod \realoptd=0$.
    We may rearrange the final condition as 
    \[
        \frac{1}{2}M_{\realopty,\realoptd}(\Delta y+\Delta d) = \mu e - \frac{1}{2}(\realopty + \Delta y) \jprod \Delta d - \frac{1}{2}\Delta y \jprod (\realoptd + \Delta d).
    \]
    This simply says that
    \[
        \frac{1}{2}\left(
            M_{\realopty,\realoptd}+M_{\optyMU,\optdMU}
        \right)
        (\Delta y+\Delta d)
        = \mu e.
    \]

    From \cite[Corollary 4.9]{faybusovich1997euclidean} we know that the operator $M_{\realopty,\realoptd}$ is invertible when the solution $(\realopty,\realoptd)$ is strictly complementary and both primal and dual non-degenerate.
    Moreover, for $(\optyMU, \optdMU)$ satisfying \eqref{eq:ip-system-reg}, we know from \cite[Corollary 4.6]{faybusovich1997euclidean} that $M_{\optyMU, \optdMU}$ is invertible. In fact, both $M_{\optyMU, \optdMU}$ and $M_{\realopty, \realoptd}$ are positive definite: in both cases, $(y, d)=(\realopty, \realoptd)$, and $(y, d)=(\optyMU, \optdMU)$, the map $m(\zeta) \defeq \iprod{\zeta}{M_{y,d}\zeta}$ is continuous on $\J$, while $m(\eta)>0$ and $m(\xi)>0$ for all $\eta \in \nullspace{A}$ and $\xi \in \range{A^*}$. For $(y,d)=(\realopty, \realoptd)$ the positivity follows from the assumed primal and dual non-degeneracy, as the operators $L(\realopty)$ and $L(\realoptd)$ are positive semi-definite. For $(y, d)=(\optyMU, \optdMU) \in \iK \times \iK$, the operators $L(\optyMU)$ and $L(\optdMU)$ are positive definite; see \cref{sec:symmc}\ref{item:symmc:ik-l}. 
    By an interpolation argument, a contradiction to invertibility would therefore be reached if $M_{y,d}$ were not positive semi-definite on the whole space \cite[cf.][proof of Lemma 32]{as-2003}.

    As a sum of invertible positive definite operators, it now follows that $M_{\realopty,\realoptd}+M_{\optyMU,\optdMU}$ is invertible.
    Consequently we estimate
    \begin{equation*}
        \begin{split}
        \norm{\Delta y}
        \le
        \norm{\Delta y+\Delta d}
        &
        =
        2\mu \norm{e} \norm{\inv{(M_{\realopty,\realoptd}+M_{\optyMU,\optdMU})}}
        \\
        &
        \le
        \frac{2\mu \sqrt{r}}{\eigval_{\min}(M_{\realopty,\realoptd}+M_{\optyMU,\optdMU})}
        \le
        \frac{2\mu \sqrt{r}}{\eigval_{\min}(M_{\realopty,\realoptd})},
        \end{split}
    \end{equation*}
    where the first inequality holds by the orthogonality of $\Delta y$ and $\Delta d$.
    The claim follows.
\end{proof}

%%%
\subsection{Strong monotonicity of the barrier}
%%%

If the barrier $B(y)=-\log(\det y)$ is as in \cref{sec:symmc}, then in the next lemma $d=-\grad B(y)$.
Therefore, the lemma provides an estimate of strong monotonicity of the gradient of the barrier.

\begin{lemma}
    \label{lemma:barrier-strong-monotonicity}
    Let $y, y' \in \iK$, and denote $d \defeq \inv y$, and $d' \defeq \inv{(y')}$.
    Then
    \begin{equation}
        \label{eq:strong-monotone-interior}
        - \iprod{d'-d}{y' - y}
        \ge
        \frac{1}{\eigval_{\max}(y')\eigval_{\max}(y)}\norm{y' - y}^2.
    \end{equation}
\end{lemma}

\begin{proof}
    There exists a unique $w \in \iK$ s.t. $d'=\inv Q_w y$ and $d=\inv Q_w y'$;
    see, e.g., \cite[Corollary 3.1]{nesterov-todd}.
    %Since by assumption $\SigmaTest \inv Q_w=\SigmaTest^{1/2} \inv Q_w \SigmaTest^{1/2}$,
    We thus see \eqref{eq:strong-monotone-interior} to hold if
    \begin{equation}
        \label{eq:sigmatest-invqw-m-estim}
        \inv Q_w \ge \frac{1}{\eigval_{\max}(y')\eigval_{\max}(y)}.
    \end{equation}
    In fact, $w$ is given by the Nesterov--Todd direction
    \begin{equation}
        \label{eq:w-transform-nt}
        w=\left(Q_{y^{-1/2}}(Q_{y^{1/2}}d')^{1/2}\right)^{-1}.
    \end{equation}
    Indeed, using the fundamental formula for quadratic presentations (\cref{sec:eucj}\ref{item:eucj-fundamental}), we see
    \begin{equation}
        \label{eq:inv-wq-1}
        \inv Q_w
        =Q_{\inv w}
        =Q_{Q_{y^{-1/2}}(Q_{y^{1/2}}d')^{1/2}}
        =Q_{y^{-1/2}}Q_{Q_{y^{1/2}}d'}^{1/2}Q_{y^{-1/2}}.
    \end{equation}
    Following \cite[p.42]{alizadeh2003soc}, from this we quickly compute
    \[
        \inv Q_w y=Q_{y^{-1/2}}Q_{Q_{y^{1/2}}d'}^{1/2}e=Q_{y^{-1/2}}Q_{y^{1/2}}d'=d'.
    \]
    Inverting $d'=\inv Q_w y$, we get $\inv{(d')}=y'=\inv{(\inv Q_v y)}=Q_v \inv y=Q_v d$. Hence $d=\inv Q_v y$. This establishes the claimed properties of $w$.

    %By uniqueness of $w$, the formula \eqref{eq:w-transform-nt} holds.
    Continuing from \eqref{eq:inv-wq-1}, we also have
    \begin{equation}
        \label{eq:invgw-expand}
        \inv Q_w=Q_{y^{-1/2}}[Q_{y^{1/2}}Q_{d'}Q_{y^{1/2}}]^{1/2}Q_{y^{-1/2}}       
    \end{equation}
    %We have \cite[Theorem V.5]{koecher-notes}
    From \cref{sec:eucj}\ref{item:eucj-power} and \eqref{eq:eigval-q-bound}, we observe that
    $
        Q_{d'}=\inv Q_{y'} \ge \eigval_{\max}(y')^{-2} I
    $.
    Thus
    \begin{equation}
        \label{eq:invqw-lowerbound}
        \inv Q_w
        \ge \frac{1}{\eigval_{\max}(y')} 
            Q_{y^{-1/2}}[Q_y]^{1/2}Q_{y^{-1/2}}
        = \frac{1}{\eigval_{\max}(y')} 
            Q_{y^{-1/2}}
        \ge \frac{1}{\eigval_{\max}(y')\eigval_{\max}(y)}.
    \end{equation}
    This proves \eqref{eq:sigmatest-invqw-m-estim} and consequently \eqref{eq:strong-monotone-interior}.
\end{proof}

We now extend the estimate to the boundary of $\K$ with a penalty using the approximations form \cref{sec:approx}.

\begin{lemma}
    \label{lemma:barrier-strong-monotonicity-boundary-ext}
    Let $y, d \in \iK$ and $\realopty, \realoptd \in \K$ with $d = \inv y$, and $\realopty \jprod \realoptd = 0$. Suppose there exist $\esty,\estd \in \K$ such that
    \begin{equation}
    	\label{eq:barrier-strong-monotonicity-boundary-ext-ass}
    	\iprod{\realoptd-\estd}{y-\realopty}=0
    	\quad\text{and}\quad \esty \jprod \estd = e.
    \end{equation}
    Then for any $\alpha \in (0, 1)$ and any $a \in \iK$ holds
    \begin{equation}
        \label{eq:barrier-strong-monotonicity-boundary-ext}
        -\iprod{d-\realoptd}{y-\realopty}
        \ge
        \frac{1-\alpha}{\eigval_{\max}(\tilde y)\eigval_{\max}(\estyTrans)}\norm{y-\realopty}_{Q_a}^2
        -\frac{\eigval_{\max}(d)\eigval_{\max}(\estd)}{4\alpha}\norm{\esty-\realopty}^2,
    \end{equation}
    where $\tilde y \defeq Q_a^{1/2} y$, and $\estyTrans \defeq Q_a^{1/2} \esty$.
\end{lemma}

\begin{proof}
    Let $Q_w$ be as in the proof of \cref{lemma:barrier-strong-monotonicity}.    
    \begin{equation}
    	\label{eq:barrier-strong-monotonicity-boundary-ext0}
    	\def\arraystretch{1.5}
        \begin{array}{rcl}
        -\iprod{d-\realoptd}{y-\realopty}
        &
        \overset{\eqref{eq:barrier-strong-monotonicity-boundary-ext-ass}}{=}
        &
        -\iprod{d-\estd}{y-\realopty}
        %\\
        %&
        =
        \iprod{y-\esty}{y-\realopty}_{\inv Q_w}
        \\
        &
        =
        &
        \iprod{y-\realopty}{y-\realopty}_{\inv Q_w}
        +\iprod{\realopty-\esty}{y-\realopty}_{\inv Q_w}
        \\
        &
        \ge
        &
        (1-\alpha)\norm{y-\realopty}_{\inv Q_w}^2
        -\frac{1}{4\alpha}\norm{\esty-\realopty}_{\inv Q_w}^2.
        \end{array}
    \end{equation}
    In the final step we have used Cauchy's inequality.   

    Let $\utilde w \defeq Q_{a^{1/2}} w$.
    By the fundamental formula of quadratic presentations (\cref{sec:eucj}\ref{item:eucj-fundamental}), 
        $$
        \inv Q_w
        =
        Q_a^{1/2}\inv Q_{Q_a^{1/2} w}Q_a^{1/2}=Q_a^{1/2}\inv Q_{\utilde w}Q_a^{1/2}.
    $$
	We also observe using fundamental formula of quadratic presentations that $\utilde w$ is $w$ from \eqref{eq:w-transform-nt} computed with the transformed variables $\tilde y=Q_a^{1/2} y$ and $\estdTrans=Q_{a^{-1/2}} \estd$.
    We therefore estimate $\inv Q_{\tilde w}$ as in \eqref{eq:invqw-lowerbound}.
    Since \eqref{eq:invgw-expand} implies
    \[
        \inv Q_w=Q_{d^{1/2}}[Q_{d^{-1/2}}Q_{d'}Q_{d^{-1/2}}]^{1/2}Q_{d^{1/2}},
    \]
    we also estimate $\inv Q_w \le \eigval_{\max}(d') \eigval_{\max}(d)$.
    Thus \eqref{eq:barrier-strong-monotonicity-boundary-ext} follows from \eqref{eq:barrier-strong-monotonicity-boundary-ext0}.
\end{proof}

\begin{lemma}
    \label{lemma:barrier-strong-monotonicity-boundary-ext-mu}
    Let $y, d \in \iK$ and $\realopty, \realoptd \in \K$ with $u \jprod d = \mu e$ for some $\mu>0$, and $\realopty \jprod \realoptd = 0$. Suppose there exist $\esty,\estd \in \K$ such that $\iprod{\realoptd-\estd}{y-\realopty}=0$ and $\esty \jprod \estd= \mu e$. Then for any $\alpha \in (0, 1)$ holds
    \begin{equation}
        \label{eq:barrier-strong-monotonicity-boundary-ext-mu}
        -\iprod{d-\realoptd}{y-\realopty}
        \ge
        \frac{(1-\alpha)\mu}{\eigval_{\max}(\tilde y)\eigval_{\max}(\estyTrans)}\norm{y-\esty}_{Q_a}^2
        -\frac{\eigval_{\max}(d)\eigval_{\max}(\estd)}{4\alpha\mu}\norm{\esty-\realopty}^2.
    \end{equation}
\end{lemma}

\begin{proof}
    We apply \cref{lemma:barrier-strong-monotonicity-boundary-ext} with $\realoptd$, $d$, and $\estd$ replaced by $\realoptd/\mu$, $d/\mu$, and $\estd/\mu$.
    This causes the right-hand-side of the estimate \eqref{eq:barrier-strong-monotonicity-boundary-ext} to be multiplied by $\mu$, along with both $\eigval_{\max}(d)$ and $\eigval_{\max}(\estd)$ to be divided by $\mu$.
\end{proof}

Applied to solutions of \eqref{eq:ip-system-reg}, we can estimate $\eigval_{\max}(y)$ and $\eigval_{\max}(y')$.

\begin{proposition}
    \label{proposition:barrier-strong-monotonicity-sclp}
    %Take $\mu=1$, and $A=\iprod{a}{\freevar}$ for $a \in \iK$, and
    Suppose $Ay=b$ implies $\iprod{a}{y}=b_0$ for some $a \in \iK$ and $b_0>0$.
    %, and that there exists some $y^* \in \iK$ with $Ay^*=b$.
    Fix $\mu > 0$, and let $(y, d, z) \in \iK \times \iK \times \R^k$ solve \eqref{eq:ip-system-reg}. Likewise, suppose  $(\optyMU, \optdMU, \optzMU) \in \iK \times \iK \times \R^k$ solves \eqref{eq:ip-system-reg} for $c=\hat c$, where $(\realopty, \realoptd, \realoptz)$ solves \eqref{eq:ip-system} for $c=\hat c$.
    % where both $\iprod{c}{\inv a}=0$ and $\iprod{c'}{\inv a}=0$.\todo{Needed?}
    If $\realopty$ and $\realoptd$ are strictly complementary, $\realoptd$ dual non-degenerate, and $\realopty$ primal non-degenerate, then for any $\alpha \in (0, 1)$ holds
    \begin{equation}
        \label{eq:barrier-strong-monotonicity-sclp}
        - \iprod{d-\realoptd}{y - \realopty}
        \ge
        \frac{(1-\alpha)\mu}{b_0^2}\norm{y - \realopty}_{Q_a}^2
        -\frac{C_{c,\mu} C_{\hat c, \mu}r}{\alpha\eigval_{\min}(M_{\realopty,\realoptd})^2}\mu,
        %-\frac{\eigval_{\max}(d)\eigval_{\max}(\optdMU)}{4\alpha}\left(
        %    \frac{2\mu \sqrt{r}}{\eigval_{\min}(M_{\realopty,\realoptd})}
        %\right),
    \end{equation}
    where for some fixed $y^* \in \iK$ with $Ay^*=b$ the constants
    \begin{equation}
    	\label{eq:ccmu}
        C_{c,\mu} \defeq \frac{\mu r + 2b_0 \norm{c}_{\inv Q_a}}{\eigval_{\min}(y^*)}.
    \end{equation}
\end{proposition}

\begin{proof}
    We begin by applying \cref{lemma:barrier-strong-monotonicity-boundary-ext-mu} with $(\esty, \estd)$ set to the $\mu$-approximation $(\optyMU, \optdMU)$ to $(\realopty,\realoptd)$ provided by \cref{lemma:central-path-approximation}. Inserting \eqref{eq:central-path-approximation} into \eqref{eq:barrier-strong-monotonicity-boundary-ext-mu}, we therefore obtain
    \begin{equation}
        \label{eq:barrier-strong-monotonicity-sclp0}
        - \iprod{d-\realoptd}{y - \realopty}
        \ge
        \frac{(1-\alpha)\mu}{\eigval_{\max}(\tilde y)\eigval_{\max}(\optyMUtrans)}\norm{y-\optyMU}_{Q_a}^2
        -\frac{\mu\eigval_{\max}(d)\eigval_{\max}(\optdMU) r}{\alpha\eigval_{\min}(M_{\realopty,\realoptd})^2}.
    \end{equation}
    It remains to estimate the eigenvalues in this expression.

    First of all, we easily derive the necessary bounds on $\lambda_{\max}(\tilde y)$ and $\lambda_{\max}(\esty)$ as
    \begin{equation}
    	\label{eq:barrier-strong-monotonicity-sclp-lambday}
    	\lambda_{\max}(\tilde y)
    	\le \tr(\tilde y)
    	=\iprod{e}{\tilde y}
    	=\iprod{a}{y}=b_0.
    \end{equation}

    Secondly, regarding the estimate on $\eigval_{\max}(d)$, we fix some $y^* \in \iK$ satisfying $Ay^*=b$. Such a point exist by our assumption of there existing solutions to \eqref{eq:ip-system-reg}; see also \cref{lemma:sclp-existence}.
    Since $d=A^*z+c$ for some $z \in \R^k$, and $d \jprod y = \mu e$, we then derive
    \[
        \begin{split}
        \eigval_{\min}(y^*)\eigval_{\max}(d)
        &
        \le \eigval_{\min}(y^*)\iprod{e}{d}
        \le \iprod{y^*}{d}
        = \iprod{\tilde y^*}{\utilde d}
        \\
        &
        = \iprod{\tilde y}{\utilde d} + \iprod{\tilde y^*-\tilde y}{\utilde d} 
        = \mu r + \iprod{\tilde y^*-\tilde y}{\utilde c}
        \\
        &
        \le \mu r + \norm{\utilde c}(\eigval_{\max}(\tilde y)+\eigval_{\max}(\tilde y^*))
        \le \mu r + 2b_0 \norm{\utilde c}.
        \end{split}
    \]
    In the last inequality we have used \eqref{eq:barrier-strong-monotonicity-sclp-lambday} for both $\tilde y$ and $\tilde y^*$.
    Since $y^* \in \iK$, so that $\eigval_{\min}(y^*)>0$, and $\norm{\utilde c}= \norm{c}_{\inv Q_a}$, this gives the claimed bounds on $\eigval_{\max}(d)$ and  $\eigval_{\max}(\estd)$.
\end{proof}

\begin{remark}
    \label{rem:unbounded-general}
    In \cref{proposition:barrier-strong-monotonicity-sclp}, the assumption that $Ay=b$ implies $\iprod{a}{y}=b_0$ for some $a \in \iK$ was only used to derive the bound \eqref{eq:barrier-strong-monotonicity-sclp-lambday} on the maximum eigenvalues of the transformed variable $\tilde y=Q_a^{1/2}y$. If we did not have this assumption, we could still bound the eigenvalues of the untransformed variable $y$ in a local neighbourhood of $\realopty$.
    Since the factor in front of $\norm{y - \realopty}_{Q_a}^2$ in particular would now depend on $\realopty$, doing so would, however, require a more local convergence analysis in \cref{sec:pedi}.
\end{remark}

%%%
\subsection{Strong monotonicity of the barrier in the second-order cone}
%%%

In the second-order cone $\K=\K_\soc \subset \E_{1+m}$, under suitable constraints $Ay=b$, we have a stronger result.

\begin{lemma}
    \label{lemma:barrier-strong-monotonicity-soc}
    Suppose $y, y', d, d' \in \iK_\soc$ with $y \circ d=y' \circ d'=\mu e$ for given $\mu>0$. Then
    \begin{equation}
        \label{eq:barrier-strong-monotonicity-soc}
        -\iprod{d-d'}{y-y'}_\J
        \ge
        \frac{\det(d)+\det(d')}{\mu}\norm{y - y'}^2_{-R},
    \end{equation}
    where $\norm{y - y'}^2_{-R} \defeq \norm{\basepart y-\basepart y'}_{\R^m}^2-(y_0-y_0')^2=-\det(y-y')$.
\end{lemma}

\begin{proof}
    We have $d=\mu Ry/\det(y)=\inv\mu \det(d) Ry$. Likewise $d'=\inv\mu \det(d') Ry'$.
    We write for brevity $\beta \defeq \inv\mu \det(d)$ and $\beta' \defeq \inv\mu \det(d')$.
    Then
    \begin{equation*}
        -\iprod{d-d'}{y-y'}_\J
        =
        - \iprod{\beta Ry-\beta' R y'}{y-y'}_\J
        =
        2\iprod{\beta y - \beta' y'}{y-y'}_{-R},
    \end{equation*}
    where the second ``inner product'' is $\iprod{x}{y}_{-R} \defeq -\iprod{Rx}{y}_{\R^{1+m}}$.
    We can thus write
    \[
        -\iprod{d-d'}{y-y'}_\J
        =
        2\beta\norm{y - y'}^2_{-R}
        +2(\beta-\beta')\iprod{y'}{y-y'}_{-R}
    \]
    as well as
    \[
        -\iprod{d-d'}{y-y'}_\J
        =
        2\beta'\norm{y - y'}^2_{-R}
        +2(\beta-\beta')\iprod{y}{y-y'}_{-R}.
    \]
    Summing these two expressions we deduce
    \begin{equation}
        \label{eq:barrier-strong-monotonicity-soc-est1}
        -\iprod{d-d'}{y-y'}_\J
        =
        (\beta+\beta')\norm{y - y'}^2_{-R}
        + (\beta-\beta')(\norm{y}^2_{-R} - \norm{y'}^2_{-R}).
    \end{equation}
    Now observe that
    \[
        \norm{y}^2_{-R}=y_0^2-\norm{\basepart y}^2=-\det(y)=-\mu^2/\det(d).
    \]
    Thus
    \[
        \begin{split}
        (\beta-\beta')(\norm{y}^2_{-R} - \norm{y'}^2_{-R})
        &
        =\mu(\det(d)-\det(d'))(\inv{\det(d')}-\inv{\det(d)})
        \\
        &
        =\mu(\det(d')-\det(d))^2/(\det(d)\det(d')) > 0.
        \end{split}
    \]
    This and \eqref{eq:barrier-strong-monotonicity-soc-est1} immediately prove the claim.
\end{proof}

For solutions of \eqref{eq:ip-system-reg} with one-dimensional linear constraints, we can extend the estimate to the boundary with some penalty. For this, we first bound the determinant with the distance
\[
    D_F(w, d) \defeq \norm{Q_w^{1/2}d-\mu_{w,d}e}
    \quad\text{for}\quad \mu_{w,d}=\iprod{w}{d}/r,
    \quad  (w, d \in \K).
\]
This distance is typically used to define the so-called short-step neighbourhood of the central path \cite[see, e.g.,][]{as-2003}.

\begin{lemma}
    \label{lemma:soc-det-estim}
    Suppose $y, d \in \iK_{\soc}$ with $y \jprod d = \mu e$ and $\iprod{a}{y}=b_0$ for some $\mu,b_0>0$ and $a \in \iK_{\soc}$.
    Then
    \begin{equation}
        \label{eq:soc-det-estim}
        \frac{2\mu^2+\sqrt{2}b_0D_F(\inv a,d)\mu}{b_0^2\det(a)}
        \le \det(d)
        \le \frac{4\mu^2+\sqrt{2}b_0D_F(\inv a,d)\mu}{b_0^2\det(a)}.
    \end{equation}
    %where for $a=e$ in particular $D_F(\inv a,d)=\norm{\basepart d}_{\R^n}\sqrt{2}$ and $\det(a)=1$.
\end{lemma}

\begin{proof}
    We define $\tilde y \defeq Q_a^{1/2} y$, and $\utilde d \defeq Q_a^{-1/2} d$. Then $\iprod{e}{\tilde y}=\iprod{a}{y}=b_0$, and by \cite[Lemma 28]{as-2003}, $\tilde y \jprod \utilde d=\mu e$.
    These conditions expand to $\tilde y_0\utilde d_0+\tilde{{\basepart y}}^T\utilde{{\basepart d}}= \mu$,  $\tilde y_0\utilde{{\basepart d}} + \utilde d_0\tilde{{\basepart y}}=0$, and $2\tilde y_0=b_0$. (In the latter, recall that the $\E_{1+m}$-inner product satisfies $\iprod{e}{\tilde y}=2e^T\tilde y$.)
    We reduce this system to $\utilde d_0^2-\norm{\utilde{{\basepart d}}}^2-2\utilde d_0\mu/b_0=0$, from where we solve
    \begin{equation}
    	\label{eq:d0-solution}
        \utilde d_0=\frac{\mu+\sqrt{\mu^2+b_0^2\norm{\utilde{{\basepart d}}}^2}}{b_0}.
    \end{equation}
    Thus
    \[
        \det(\utilde d)
        =\utilde d_0^2-\norm{\utilde{{\basepart d}}}^2
        =\frac{2\mu^2+2\mu\sqrt{\mu^2+b_0^2\norm{\utilde{{\basepart d}}}^2}}{b_0^2},
    \]
    from which we easily estimate
    \begin{equation}
        \label{eq:soc-det-estim-e}
        \frac{2\mu^2+2\mu b_0\norm{\utilde{{\basepart d}}}}{b_0^2}
        \le \det(\utilde d)
        \le \frac{4\mu^2+2\mu b_0\norm{\utilde{{\basepart d}}}}{b_0^2}.
    \end{equation}
    %Observe that this directly gives the special case estimate for $a=e$.

    To finish deriving \eqref{eq:soc-det-estim}, from \cref{sec:eucj}\ref{item:det-Qxy} we recall that
    $
        \det(\utilde d)=\det(a)\det(d)
        %=\eigval_{\min}(a)\eigval_{\max}(a)\det(d).
    $.
    We also have $r\utilde d_0=\iprod{\utilde d}{e}=\iprod{d}{\inv a}$ for the rank $r=2$, so
    \begin{equation}
        \label{eq:soc-d-a-distance}
        \sqrt{2} \norm{\utilde{{\basepart d}}}_{\R^n}
        =\norm{\utilde d - \utilde d_0 e}_{\J}
        %=r^{-1/2}\norm{Q_a^{-1/2}d-r^{-1/2}\iprod{d}{a}e}
        =\norm{Q_a^{-1/2}d-\mu_{\inv a,d}e}_{\J}
        =D_F(\inv a,d),
        %=r^{-1/2}\norm{Q_a^{-1/2}(d-r^{-1/2}\iprod{d}{a} a)}
    \end{equation}
    where we emphasise the standard Euclidean norm on $\utilde{{\basepart d}} \in \R^n$ versus the $\sqrt{2}$-scaled standard norm on $\J$.
    With this, \eqref{eq:soc-det-estim-e} gives \eqref{eq:soc-det-estim}.
\end{proof}

If $D_F(\inv a, \realoptd)>0$, or alternatively $\det(\realopty)>0$, then the next proposition shows local strong monotonicity of the barrier for $d$ close to $\realoptd$ and $\mu>0$ small. Moreover, if $D_F(\inv a, \realoptd)>0$, the factor of strong monotonicity does not vanish as $\mu \downto 0$.

\begin{proposition}
    \label{proposition:barrier-strong-monotonicity-soc-extension}
    Let $\K=\K_\soc$, and suppose $Ay=b$ implies $\iprod{a}{y}=b_0$ for some $a \in \iK$ and $b_0>0$. 
    %Suppose moreover that $c,\realopt{c} \in \nullspace{A}$.
    Let $(y, d, z) \in \iK \times \iK \times \R^k$ solve \eqref{eq:ip-system-reg}, and likewise that  $(\realopty, \realoptd, \realoptz) \in \K \times \K \times \R^k$ solve \eqref{eq:ip-system} for $c=\realopt{c}$.
    Then
    \begin{multline}    
        \label{eq:barrier-strong-monotonicity-soc-extension}
        -\iprod{d-\realoptd}{y-\realopty}
        \ge
        \frac{\mu+2^{-1/2}b_0[D_F(\inv a,d)+D_F(\inv a,\realoptd)]}{b_0^2/2}\norm{y - \realopty}_{Q_a}^2
        - \mu
        \\
        +\frac{2^{-1/2}b_0D_F(\inv a,\realoptd)\mu}{2\mu+2^{-1/2}b_0D_F(\inv a,d)}
        + \frac{\mu+2^{-1/2}b_0D_F(\inv a,d)}{b_0^2/2}\det(Q_a^{1/2}\realopty).
    \end{multline}
\end{proposition}

\begin{proof}
    We have
    \begin{equation}
        \label{eq:soc-realopt-prod}
        0 = \realopty \jprod \realoptd=(\realopty_0\realoptd_0 + \basepartRealopty^T\basepartRealoptd, \realopty_0\basepartRealoptd+\realoptd_0\basepartRealopty).
    \end{equation}
    Since $\iprod{a}{\realopty}=b_0 > 0$, and $\realopty \in \K$, necessarily $\realopty_0>0$. Since, moreover, $\realopty \ne 0$, we cannot have $\realoptd \in \iK$ for $\realopty \jprod \realoptd = 0$ to hold. Therefore $0=\det(\realoptd)=\realoptd_0^2-\norm{\basepartRealoptd}^2$. 
    It follows from \eqref{eq:soc-realopt-prod} that $\realoptd=\realopt{\beta} R\realopty$ for
    \begin{equation}
        \label{eq:realopt-beta-expr}
        \realopt{\beta}
        =-\frac{ \basepartRealopty^T\basepartRealoptd }{ \realopty_0^2 }
        = \frac{\realoptd_0}{\realopty_0}
        = \frac{\norm{\basepartRealoptd}_{\R^m}}{ \realopty_0 }
        \ge 0.
    \end{equation}
    We may therefore repeat the steps of  \cref{lemma:barrier-strong-monotonicity-soc} until \eqref{eq:barrier-strong-monotonicity-soc-est1} to obtain
    \begin{equation}
        \label{eq:barrier-strong-monotonicity-soc-extension-est1}
        -\iprod{d-\realoptd}{y-\realopty}
        =
        (\beta+\realopt{\beta})\norm{y - \realopty}^2_{-R}
        + (\beta-\realopt{\beta})(\norm{y}^2_{-R} - \norm{\realopt{y}}^2_{-R}).
    \end{equation}

    We have $\det(\realopty)=-\norm{\realopty}^2_{-R}=\realopty_0^2-\norm{\basepartRealopty}^2 \ge 0$. If this is non-zero, $\realopty \in \iK$. But in that case $\realopty \jprod \realoptd = 0$ implies $\realoptd=0$, and consequently $\realopt{\beta}=0$.
    Thus $\realopt{\beta}\norm{\realopty}^2_{-R}=0$ whether or not $\norm{\realopty}^2_{-R}=0$.
    Using $\norm{y}_{-R}^2=-\det(y)=-\mu^2/\det(d)$ and $\beta=\det(d)/\mu$, we therefore obtain from \eqref{eq:barrier-strong-monotonicity-soc-extension-est1} that
    \begin{equation}
        \label{eq:barrier-strong-monotonicity-soc-extension-est2}
        -\iprod{d-\realoptd}{y-\realopty}
        =
        (\inv\mu\det(d)+\realopt{\beta})\norm{y - \realopty}^2_{-R}
        -\mu
        %+\frac{-\mu+\realopt{\beta}\mu^2/\det(d)}{2}.
        +\frac{\realopt{\beta}\mu^2}{\det(d)}
        + \frac{\det(d)\det(\realopty)}{\mu}.
    \end{equation}

    If $a=e$, we have $y_0=\realopty_0=b_0/2$, so that $2\norm{y-\realopty}_{-R}^2=\norm{y-\realopty}_\J^2$.
    Reasoning as in \eqref{eq:soc-d-a-distance}, \eqref{eq:realopt-beta-expr} gives $\realopt{\beta} = \sqrt{2}D_F(\inv a, \realoptd)/b_0=\sqrt{2}D_F(e, \realoptd)/b_0$.
    With the help of \cref{lemma:soc-det-estim}, \eqref{eq:barrier-strong-monotonicity-soc-extension-est2} thus yields
    \begin{equation}
        \label{eq:barrier-strong-monotonicity-soc-extension-est3}
        \begin{split}
        -\iprod{d-\realoptd}{y-\realopty}
        &
        \ge
        \frac{2\mu+\sqrt{2}b_0[D_F(e,d)+D_F(e,\realoptd)]}{b_0^2}\norm{y - \realopty}^2
        -\mu
        \\
        & \phantom{ = } 
        %+\frac{\sqrt{2}b_0D_F(e,\realoptd)\mu}{4\mu+\sqrt{2}b_0D_F(e,d)}
        + \frac{2\mu+\sqrt{2}b_0D_F(e,d)}{b_0^2}\det(\realopty),
        \end{split}
    \end{equation}
    where we have entirely eliminated the term $\realopt{\beta}\mu^2/\det(d) \ge 0$.
    Since $\eigval_{\min}(e)=\det(e)=1$, the estimate \eqref{eq:barrier-strong-monotonicity-soc-extension} is immediate in the case $a=e$.

    If $a \ne e$, we define $\tilde y \defeq Q_a^{1/2} y$, and $\utilde d \defeq Q_a^{-1/2} d$ as in \cref{lemma:soc-det-estim}.
    Then $(\tilde y, \utilde d, z)$ continues to satisfy \eqref{eq:ip-system-reg} with $A$ replaced by $\tilde A \defeq A Q_a^{-1/2}$ and $\utilde c \defeq Q_a^{-1/2} c$.
    The same holds with \eqref{eq:ip-system} for $\realoptyTrans \defeq Q_a^{1/2} \realopty$ and $\realoptdTrans \defeq Q_a^{-1/2} \realoptd$.
    Therefore, \eqref{eq:barrier-strong-monotonicity-soc-extension-est3} holds for these transformed variables. Since $D_F(e,\utilde d)=D_F(\inv a, d)$, 
    %$\det(\tilde{\realopty})=\det(a)\det(\realopty)$,
    as well as
    $
        \norm{\tilde y - \realoptyTrans}^2=\norm{y-\realopty}^2_{Q_a}
        % \ge \norm{y-\realopty}^2 \cdot \eigval_{\min}(a)^2,
    $,
    and $-\iprod{d-d'}{y-y'}=-\iprod{\utilde d-\realoptdTrans}{\tilde y-\realoptyTrans}$, we obtain the claim.
\end{proof}

\begin{corollary}
    \label{corollary:barrier-strong-monotonicity-soc-extension}
    %Let $\K=\K_\soc$, and suppose $Ay=b$ implies $\iprod{a}{y}=b_0$ for some $a \in \iK$ and $b_0>0$. 
    Let $\K=\K_\soc$, and suppose $A=\iprod{a}{\freevar}$ for some $a \in \iK$.
    Suppose moreover that $\iprod{\inv a}{c}=\iprod{\inv a}{\realopt{c}}=0$.
    Let $(y, d, z) \in \iK \times \iK \times \R^k$ solve \eqref{eq:ip-system-reg}, and likewise that  $(\realopty, \realoptd, \realoptz) \in \K \times \K \times \R^k$ solve \eqref{eq:ip-system} for $c=\realopt{c}$.
    If $\realopt{c} \ne 0$, then
    \begin{equation}
        \label{eq:barrier-strong-monotonicity-soc-extension-estimate-nz}
        \begin{split}    
        -\iprod{d-\realoptd}{y-\realopty}
        &
        \ge
        \frac{\mu+2^{-1/2}b_0[\norm{c}_{\inv Q_a}+\norm{\realopt{c}}_{\inv Q_a}]}{b_0^2/2}\norm{y - \realopty}_{Q_a}^2
        -\mu.
        %\\ & \phantom{ \ge }
        %-\epsilon\norm{c-\realopt{c}}_{\inv Q_a}^2 - 4\left[1+\frac{2(%4\mu+\sqrt{2}b_0\norm{c}_{\inv Q_a})}{\sqrt{2}b_0\epsilon}\right]\mu^2.
        \end{split}
    \end{equation}
    Otherwise, if $\realopt{c} = 0$ with $\realopty=b \inv a/2$, then
    \begin{equation}    
        \label{eq:barrier-strong-monotonicity-soc-extension-estimate-z}
        -\iprod{d-\realoptd}{y-\realopty}
        \ge
        \frac{\mu+2^{-1/2}b_0 \norm{c}_{\inv Q_a}}{b_0^2/2}\norm{y - \realopty}_{Q_a}^2.
    \end{equation}
\end{corollary}

We say that \eqref{eq:barrier-strong-monotonicity-soc-extension-estimate-nz} is strong monotonicity of the barrier ``with a penalty'', $\mu$.

\begin{proof}
    We do not until the very end of the proof use the assumption $A=\iprod{a}{\freevar}$. For now, we use the weaker assumption that $Ay=b$ implies $\iprod{a}{y}=b_0$.
    %, and hence write $b_0$ instead of $b$.
    We apply \cref{proposition:barrier-strong-monotonicity-soc-extension}.
    This gives
    \begin{multline}    
        \label{eq:barrier-strong-monotonicity-soc-extension-fubar}
        -\iprod{d-\realoptd}{y-\realopty}
        \ge
        \frac{\mu+2^{-1/2}b_0[D_F(\inv a,d)+D_F(\inv a,\realoptd)]}{b_0^2/2}\norm{y - \realopty}_{Q_a}^2
        - \mu
        \\
        +\frac{2^{-1/2}b_0D_F(\inv a,\realoptd)\mu}{2\mu+2^{-1/2}b_0D_F(\inv a,d)}
        + \frac{\mu+2^{-1/2}b_0D_F(\inv a,d)}{b_0^2/2}\det(Q_a^{1/2}\realopty).
    \end{multline}

    If $D_F(\inv a, \realoptd)=0$, by assumption $\realopty=2 b_0 \inv a$. This implies $\det(Q_a^{1/2} \realopty)=b_0/2$. Consequently
    \[
        \frac{\mu+2^{-1/2}b_0D_F(\inv a,d)}{b_0^2/2}\det(Q_a^{1/2}\realopty)
        \ge \mu.
    \]    
    Therefore no penalty is imposed, and \eqref{eq:barrier-strong-monotonicity-soc-extension-fubar} reduces to
    \begin{equation}    
        \label{eq:barrier-strong-monotonicity-soc-extension-fubar-z}
        -\iprod{d-\realoptd}{y-\realopty}
        \ge
        \frac{\mu+2^{-1/2}b_0 D_F(\inv a,d)}{b_0^2/2}\norm{y - \realopty}_{Q_a}^2.
    \end{equation}

    Suppose then that $D_F(\inv a, \realoptd)>0$. On the right hand side of \eqref{eq:barrier-strong-monotonicity-soc-extension-fubar}, only the term $-\mu$ is negative.
    Thus the condition holds if
    \begin{equation}
        \label{eq:barrier-strong-monotonicity-soc-extension-fubar-nz}
        -\iprod{d-\realoptd}{y-\realopty}
        \ge
        \frac{\mu+2^{-1/2}b_0[D_F(\inv a,d)+D_F(\inv a,\realoptd)]}{b_0^2/2}\norm{y - \realopty}_{Q_a}^2
        -\mu.
    \end{equation}    

    Finally, using our assumptions that $A=\iprod{a}{\freevar}$ and $\iprod{\inv a}{c}=0$, we have $d=za+c$ and $\mu_{\inv a,d}=\iprod{\inv a}{d}/r=z$ for some $z \in \R$. Thus
    \begin{equation}
        \label{eq:df-qa}
        D_F(\inv a, d)=\norm{Q_a^{-1/2}(d-za)}=\norm{c}_{Q_a^{-1}}.
    \end{equation}
    Likewise $D_F(\inv a, \realoptd)=\norm{\realopt{c}}_{Q_a^{-1}}$.
    Therefore, the cases $D_F(\inv a, \realoptd)>0$ and $D_F(\inv a, \realoptd)=0$ are equivalent to the cases on $\norm{\realopt{c}}$ in the statement of the corollary.
    Inserting \eqref{eq:df-qa} into \eqref{eq:barrier-strong-monotonicity-soc-extension-fubar-z} consequently yields the claimed estimates.
\end{proof}

\begin{remark}
    Recall \cref{rem:unbounded-general} on removing the assumption on the existence of $a \in \iK$ such that $\iprod{a}{y}=b_0$.
    In the proof of \cref{proposition:barrier-strong-monotonicity-soc-extension}, this assumption was not used until the derivation of \eqref{eq:barrier-strong-monotonicity-soc-extension-est3} from  \eqref{eq:barrier-strong-monotonicity-soc-extension-est2}.
    At that point, we used this fact to ensure that $\iprod{a}{y-\realopty}=0$ and, in particular, that $\norm{y-\realopty}^2_{-R}=\norm{y-\realopty}^2 \ge 0$ when $a=e$ after transformation.
    Could we still get our overall estimates without this assumption?

    On the two-dimensional Jordan algebra $\E_{1+1}$, pick $a=(a_0, \bar a) \not\in \K$, $b \in \R$, and set $Ay \defeq a_0y_0+\bar a\bar y$. Without loss of generality, by negating both $a$ and $b$.
    Assume that $\bar a<0$.
    Then $y$ satisfying $Ay=b$ has the form $y=\theta v + (b/a_0) e$ for $v=(-\bar a, a_0)$ and some $\theta \in \R$. Since $a \not \in \K$ and $\bar a<0$, we have $-\bar a = \abs{\bar a} > a_0$. Consequently, $v \in \iK$.

    Now, with $y=\theta v + (b/a_0)e$ and $\realopty=\realopt\theta v + (b/a_0)e$ with $\theta \ne \theta^*$, we have
    \[
        \norm{y-\realopty}^2_{-R}=\norm{(\theta-\theta^*)v}_{-R}^2=-(\theta-\theta^*)^2\det(v) < 0.
    \]
    This implies that the first term in \eqref{eq:barrier-strong-monotonicity-soc-extension-est2} is negative for all the feasible points in every neighbourhood of $\realopty$.
    %Moreover, since for large enough $\theta, \theta^*$, both $y, \realopty \in \iK$, we can apply \cref{lemma:soc-det-estim} with $a=e$ to obtain that $\det(d) \le C \mu$ for some $C>0$.
    This seems at first a negative result.
    If, however $\det(\realopty)>0$, then also $\det(\realoptd)>0$, so in a neighbourhood of $(\realopty, \realoptd)$, the last term of \eqref{eq:barrier-strong-monotonicity-soc-extension-est2} will be bounded away from zero.
    We can therefore still, \emph{locally}, obtain quadratic estimates like those in  \cref{corollary:barrier-strong-monotonicity-soc-extension}.

    On the other hand, if $\det(\realopty)=0$, we can run into difficulties.
    Consider $b=0$, so that $y=\theta v$ and $\realopty=0$.
    Then also $\realopt\beta=0$, so the right-hand-side of \eqref{eq:barrier-strong-monotonicity-soc-extension-est2} is negative, and we do not get the quadratic-penalised estimate. The solution $\realopty=0$ would, however, be primal degenerate. Indeed, in the general non-degenerate strictly complementary case, \cref{proposition:barrier-strong-monotonicity-sclp,rem:unbounded-general} still guarantee a local estimate with worse constants than the more fine-grained approach of \cref{proposition:barrier-strong-monotonicity-sclp} might provide.
\end{remark}

%%%
\section{An abstract preconditioned proximal point iteration}
\label{sec:general-estimates}
%%%

In this section, we recall some of the core results from \cite{tuomov-proxtest}.
We start by setting
\begin{equation}
    \label{eq:h}
    H(u) \defeq
        \begin{pmatrix}
            \subdiff G(x) + K^* y \\
            \subdiff F^*(y) -K x
        \end{pmatrix},
\end{equation}
and for some $\tau_i, \tauTest_i, \sigma_{i+1}, \sigmaTest_{i+1}>0$, defining the step length and ``testing'' operators
\begin{equation}
    \label{eq:test}
    \Step_{i+1} \defeq
    \begin{pmatrix}
        \tau_i I & 0 \\
        0 & \sigma_{i+1} I
    \end{pmatrix},
    \quad\text{and}\quad
    \Test_{i+1} \defeq
        \begin{pmatrix}
            \tauTest_i I & 0 \\
            0 & \sigmaTest_{i+1} I
        \end{pmatrix}.
\end{equation}
We also let $\BregFn_{i+1}: X \times Y \setto X \times Y$ for each $i \in \N$ be an abstract non-linear preconditioner, dependent on the current iterate $\thisu$.
Then we consider the generalised proximal point method, which involves solving
\begin{equation*}
    \label{eq:pp}
    \tag{PP}
    0 \in \Step_{i+1} H(\nextu) + \BregFn_{i+1}(\nextu)
\end{equation*}
for the unknown next iterate $\nextu$.
To obtain convergence rates for the resulting method, the idea from \cite{tuomov-proxtest,tuomov-cpaccel} will be to analyse the inclusion obtained after multiplying \eqref{eq:pp} by the testing operator $\Test_{i+1}$.

Assuming $G$ to be (strongly) convex with factor $\gamma>0$, we also introduce
\begin{equation*}
    \GammaLift{i+1}(\gamma) \defeq \begin{pmatrix}
        2\gamma\tau_i I & 2\tau_i K^* \\
        -2\sigma_{i+1} K & 0 %2\Sigma_{i+1} \GammaF
    \end{pmatrix},
\end{equation*}
which is an operator measure of strong monotonicity of $H$.

The next lemma, which is relatively trivial to prove \cite{tuomov-proxtest}, forms the basis from which our work proceeds.

\begin{theorem}
    \label{thm:convergence-result-main}
    Let us be given $K \in \linear(X; Y)$, $G \in \convex(X)$, and $F^* \in \convex(Y)$ on Hilbert spaces $X$ and $Y$.  
	For each $i \in \N$, for some $\AltBregFn_{i+1}: X \times Y \setto X \times Y$ and $\Precond_{i+1} \in \linear(X \times Y; X \times Y)$, take
	\begin{equation}
		\label{eq:tilde-v}
		\BregFn_{i+1}(u) \defeq \AltBregFn_{i+1}(u) + \Precond_{i+1}(u-\thisu).
	\end{equation}
	Assume that \eqref{eq:pp} is solvable,
 	%and denote the iterates by $\thisu=(\thisx, \thisy)$.
 	$\Test_{i+1}\Precond_{i+1}$ is self-adjoint, and $G$ is (strongly) convex with factor $\gamma \ge 0$.
 	If for all $i \in \N$ the estimate 
	\begin{multline}
		\label{eq:convergence-fundamental-condition-iter}
		\tag{C0-$\Gamma$}		
		\underbrace{\frac{1}{2}\norm{\nextu-\thisu}_{\Test_{i+1}\Precond_{i+1}}^2}_{\text{step length in local metric}}
	    + \underbrace{\frac{1}{2}\norm{\nextu-\realoptu}_{\Test_{i+1}(\GammaLift{i+1}(\gamma)+\Precond_{i+1})-\Test_{i+2}\Precond_{i+2}}^2}_{\text{linear preconditioner update discrepancy}}
	    %\\
	    %+\underbrace{\iprod{\subdiff G(\nextx)-\subdiff G(\optx)}{\Tau_{i}^*\TauTest_{i}^*(\nextx - \optx)}}_{\text{variably useful non-linear part from } H}
        \\
        +\underbrace{\iprod{\subdiff F^*(\nexty)-\subdiff F^*(\realopty)}{\nexty - \realopty}_{\SigmaTest_{i+1}\Sigma_{i+1}}}_{\text{variably useful remainder from } H}
	    %\\
		+ \underbrace{\iprod{\Test_{i+1} \AltBregFn_{i+1}(\nextu)}{\nextu - \realoptu}}_{\text{from non-linear preconditioner}}
		\ge
		- \Penalty_{i+1}
	\end{multline}
	holds, then
    \begin{equation}
        \label{eq:convergence-result-main}
        \frac{1}{2}\norm{u^N-\realoptu}^2_{\Test_{N+1}\Precond_{N+1}}
        \le
        \frac{1}{2}\norm{u^0-\realoptu}^2_{\Test_{1}\Precond_{1}}
        +
        \sum_{i=0}^{N-1} \Penalty_{i+1},
        \quad
        (N \ge 1).
    \end{equation}		
\end{theorem}

\begin{proof}
	This is \cite[Theorem 3.1]{tuomov-proxtest} specialised to scalar step length and testing operators $\Tau_i=\tau_i I$, $\TauTest_i=\tauTest_i I$, $\Sigma_{i+1}=\sigma_{i+1} I$, and $\SigmaTest_{i+1}=\sigmaTest_{i+1} I$, as well as $\tilde\Gamma=\gamma I$.
\end{proof}

It is possible to extend this theorem to provide an estimate on an ergodic duality gap \cite[see][Theorem 4.6]{tuomov-proxtest}.
For the sake of conciseness, we have however decided against including such estimates in the present work. For this reason, in the following, we concentrate on strongly convex $G$.

%%%
\section{A primal--dual method with a barrier preconditioner}
\label{sec:pedi}
%%%

Let $F(y) \defeq \delta_{\{A\freevar=b\}}(y) + \delta_{\K}(y)$ for some $A \in \linear(\J; Z)$, where $\J$ and $Z$ are Hilbert spaces, $\J$ also a Euclidean Jordan algebra.
Let $\K$ be the cone of squares of $\J$.
We suppose there exists some $y \in \iK$ with $Ay=b$. Then the subdifferential sum formula (see, e.g., \cite{rockafellar-convex-analysis}) applies, so that
\begin{equation}
    \label{eq:interior-fstar}
    \subdiff F^*(y)= 
    \begin{cases}
        \{A^*z \mid z \in Z\} + N_{\K}(y), & Ay=b \text{ and } y \in \K, \\
        \emptyset, & \text{otherwise}.
    \end{cases}
\end{equation}
In particular, if $y \in \iK$ with $Ay=b$, then $\subdiff F^*(y)= \{A^*z \mid z \in Z\}$. Note from \cref{sec:symmc}\ref{item:symmc-normal} and \ref{property:sc-iprod-zero} that
\begin{equation}
    \label{eq:nky}
	N_{\K}(y)=\{-d \mid d \in \K,\, p \jprod d=0 \} \quad (y \in \K).
\end{equation}
Inserting \eqref{eq:interior-fstar} into $0 \in H(\realoptu)$, the latter expands as
\[
    0 \in \subdiff G(\realoptx) + K^*\realopty,
    0 \in A^*\realoptz + N_\K(\realopty) - K\realoptx,
    Ay=b, 
    y \in \K
\]
for some $\realoptz \in Z$. Based on \eqref{eq:nky}, this
may also be written as the existence of $(\realoptx, \realopty, \realoptd, \realopt{z}) \in X \times \K \times \K \times Z$ with
\begin{equation}
    \label{eq:interior-oc}
    \tag{IOC}
    -K^*\realopty \in \subdiff G(\realoptx),
    \quad
    A\realopty=b,
    \quad
    A^*\realopt{z}-K\realoptx=\realoptd,
    \quad
    \realopty \jprod \realoptd=0.
%    \quad
%    \realopty,\realoptd \in \K.
\end{equation}

In the following, we develop an algorithm for solving this system, incorporating a barrier-based nonlinear preconditioner for dual updates.
As mentioned after \cref{thm:convergence-result-main}, for conciseness we limit our attention to strongly convex $G$, and only analyse the convergence of iterates, not the gap.
The theory from \cite{tuomov-proxtest} could be used to extend the analysis to the gap.
Moreover, following the approach of \cite{tuomov-blockcp}, it would be possible to extend our work to stochastic and ``spatially-adaptive'' updates.

%%%
\subsection{A general estimate for dual barrier preconditioning}
%%%

To construct algorithms with the help of the theory from \cref{sec:general-estimates}, we have to construct the preconditioner $\BregFn_{i+1}(\nextu) \defeq \AltBregFn_{i+1}(\nextu) + \Precond_{i+1}(\nextu-\thisu)$. We specifically take
\begin{equation}
    \label{eq:interior-precond}
    \Precond_{i+1}=\begin{pmatrix} I & 0 \\ 0 & 0 \end{pmatrix},
    \quad\text{and}\quad
    \AltBregFn_{i+1}(\nextu) = (0, \sigma_{i+1} [K(\nextx-\thisx)-\nextd]),
\end{equation}
where $\nextd \in \iK$ is defined to satisfy $\nexty \jprod \nextd = \mu_{i+1} e$ for some $\mu_{i+1}>0$.
The term $\sigma_{i+1}K(\nextx-\thisx)$ in $\AltBregFn_{i+1}$ decouples the primal and dual updates so that\eqref{eq:pp} may be written as the system
\begin{subequations}
\label{eq:alg-interior0}
\begin{align}
    0 & \in \tau_i \subdiff G(\nextx) + \tau_i K^*\nexty + (\nextx-\thisx),  \\
    0 & \in \sigma_{i+1} [ A^* \nexxt{z} - K\thisx - \nextd], \quad\text{as well as} \\
    & \phantom{ \in } \nexty \jprod \nextd =\mu_{i+1} e \quad\text{and}\quad A\nexty=b \quad\text{with}\quad \nexty, \nextd \in \iK.
\end{align}
\end{subequations}
For this general setup, we have the following lemma:

\begin{lemma}
    \label{lemma:interior0}
    Let $F^*$ have the structure \eqref{eq:interior-fstar}.
    Take $\Precond_{i+1}$ and $\AltBregFn_{i+1}$ according to \eqref{eq:interior-precond}.
    Suppose for some $\omega_{i+1}, \PenaltyX_{i+1} \in \R$ for all $i \in \N$ that
    \begin{subequations}
    \label{eq:interior0-conds}    
    \begin{align}
        \label{eq:interior0-sc-cond}
        -\iprod{\nextd-\realoptd}{\nexty - \realopty}
        &
        \ge \omega_{i+1}\norm{\nexty-\realopty}^2
        -\PenaltyX_{i+1},
        \\
        \label{eq:interior0-offdiag-cond}
        \sigmaTest_{i+1}\sigma_{i+1} & = \tauTest_{i}\tau_{i},
        \\
        \label{eq:interior0-sigmatest-lowerbound}
        2\omega_{i+1} & \ge \tau_i \norm{K}^2,
        \quad\text{and}
        \\
        \label{eq:interior0-cn-tau}
        \tauTest_{i+1} & \le \tauTest_i(1+2\tau_i\tilde \gamma).
    \end{align}  
    \end{subequations} 
    Then \eqref{eq:convergence-fundamental-condition-iter} holds with $\Penalty_{i+1}=\sigmaTest_{i+1}\sigma_{i+1}\PenaltyX_{i+1}$, and $\Test_{i+1} \Precond_{i+1}$ is self-adjoint with
    \begin{equation}
        \label{eq:interior0-zimi-estim}
        \Test_{i+1} \Precond_{i+1} =
        \begin{pmatrix}
            \tauTest_i I & 0 \\
            0 & 0
        \end{pmatrix}
        \ge 0.
    \end{equation}
\end{lemma}

\begin{proof}
    The condition \eqref{eq:convergence-fundamental-condition-iter} now reads
    \begin{multline}
        \label{eq:convergence-fundamental-condition-iter-interior0}
        \underbrace{\frac{1}{2}\norm{\nextu-\thisu}_{\Test_{i+1}\Precond_{i+1}}^2}_{\text{step in local norm}}
        +\underbrace{\frac{1}{2}\norm{\nextu-\realoptu}_{D_{i+2}}^2}_{\text{lin.~precond.~upd.~d.}}
        + \underbrace{\sigmaTest_{i+1}\sigma_{i+1}\iprod{K(\nextx-\thisx)}{\nexty - \realopty}}_{\text{de-coupling term from $\AltBregFn$}}
        \\
        + \underbrace{
        \sigmaTest_{i+1}\sigma_{i+1} \iprod{A^*(\nexxt{z}-\realoptz)}{\nexty - \realopty}
        - \sigmaTest_{i+1}\sigma_{i+1} \iprod{\nextd-\realoptd}{\nexty - \realopty}}_{\text{$F^*$ term from \eqref{eq:convergence-fundamental-condition-iter} as well as $\nextd$ from $\AltBregFn$}}
        %\\
        \ge
        - \Penalty_{i+1}
    \end{multline}
    with the linear preconditioner update discrepancy
    \begin{equation*}
        D_{i+2}
        \defeq \Test_{i+1}(\GammaLift{i+1}(\gamma)+\Precond_{i+1})-\Test_{i+2}\Precond_{i+2}.
    \end{equation*}

    The expansion and estimate \eqref{eq:interior0-zimi-estim} are trivially verified along with the self-adjointness of $\Test_{i+1}\Precond_{i+1}$.
    This expansion allows us to write
    \[
        D_{i+2}
        =
        \begin{pmatrix}
           \tauTest_i(1+2\tau_i\gamma)I-\tauTest_{i+1}I
             & 2\tauTest_i \tau_i K^* \\
            -2\sigmaTest_{i+1}\sigma_{i+1}K & 0
        \end{pmatrix}.
    \]
    We use \eqref{eq:interior0-offdiag-cond} to cancel the off-diagonals of $D_{i+2}$ in \eqref{eq:convergence-fundamental-condition-iter-interior0}.
    Then we use the fact that $A(\nexty-\realopty)=0$ to cancel the first term on the second line of \eqref{eq:convergence-fundamental-condition-iter-interior0}. Finally, we use $\Penalty_{i+1}=\sigmaTest_{i+1}\sigma_{i+1}\PenaltyX_{i+1}$ and \eqref{eq:interior0-sc-cond} to estimate the second term on the second line of \eqref{eq:convergence-fundamental-condition-iter-interior0}. This gives the condition
    \begin{multline}
        \label{eq:convergence-fundamental-condition-iter-interior2}        
        \frac{\tauTest_i}{2}\norm{\nextx-\thisx}^2
        + 
        \frac{\sigmaTest_{i+1}\sigma_{i+1}\omega_{i+1}}{2}\norm{\nexty-\realopty}^2
        +\frac{\tauTest_i(1+2\gamma\tau_i)-\tauTest_{i+1}}{2}\norm{\nextx-\realoptx}^2
        \\
        +\sigmaTest_{i+1}\sigma_{i+1}\iprod{K(\nextx-\thisx)}{\nexty - \realopty}
        \ge
        0.
    \end{multline}
    Application of \eqref{eq:interior0-cn-tau}, as well as Cauchy's inequality to the inner product term, shows that \eqref{eq:convergence-fundamental-condition-iter-interior2} and consequently \eqref{eq:convergence-fundamental-condition-iter} is satisfied if
    \[
        \sigmaTest_{i+1}\sigma_{i+1}\omega_{i+1}  \ge \frac{1}{2} \inv\tauTest_i \sigmaTest_{i+1}^2\sigma_{i+1}^2 KK^* .
    \]
    This follows from \eqref{eq:interior0-offdiag-cond} and \eqref{eq:interior0-sigmatest-lowerbound}.
\end{proof}

We define $\tau_i$ through \eqref{eq:interior0-sigmatest-lowerbound} for a lower bound $\omegalowerbound{i+1}$ of $\omega_{i+1}$.
Likewise, we take \eqref{eq:interior0-cn-tau} as an equality as the definition of $\tauTest_{i+1}$. We observe that $\sigma_{i+1}$ and $\sigmaTest_{i+1}$ are irrelevant to the algorithm in \eqref{eq:alg-interior0}, as will be the specific choice of $\tauTest_0>0$ to the satisfaction of \eqref{eq:interior0-conds}. Taking $\tauTest_0=1$, we obtain \cref{alg:alg-interior-scalar} from \eqref{eq:alg-interior0}.

\begin{algorithm}[Barrier-preconditioned primal--dual method]
    \label{alg:alg-interior-scalar}
    \mbox{}%
    \begin{algorithmic}[1]
    \Require Linear operator $K \in \linear(X; \J)$, strongly convex $G \in \convex(X)$, and $F^* \in \convex(\J)$ of the form \eqref{eq:interior-fstar}.
    Factor $\gamma>0$ of the strong convexity of $G$.
    Rules for $\mu_i, \omegalowerbound{i}>0$. 
    \State Choose initial iterates $x^0 \in X$ and $y^0 \in Y$.
    \State Set initial testing parameter $\tauTest_0 \defeq 1$.
    \Repeat
        \State\label{step:alg-interior-scalar-tau}
            Calculate $\mu_i$, $\omegalowerbound{i}$, and step length
            \[
                \tau_i \defeq 2\omegalowerbound{i+1}/\norm{K}^2.
            \]
        \State\label{step:alg-interior-scalar-tautest}
            Update testing parameter
            \[
                \tauTest_{i+1} \defeq \tauTest_i(1+2\gamma\tau_i).
            \]
        \State\label{step:alg-interior-scalar-dual} Perform dual update by solving for $(\nexty, \nextd, \nexxt{z}) \in \iK \times \iK \times Z$ the system
        \begin{equation*}
        	A\nexty = b,
        	\quad
            A^*\nexxt{z} - K\thisx=\nextd,
            \quad\text{and}\quad
            \nexty \jprod \nextd =\mu_{i+1} e.
        \end{equation*}
        \State\label{step:alg-interior-scalar-primal} Perform primal update
            \[
                \nextx \defeq \inv{(I+ \tau_i \subdiff G)}(\thisx - \tau_i K^*\nexty).
            \]        
    \Until a stopping criterion is satisfied.
    \end{algorithmic}%
\end{algorithm}

\begin{remark}[Solution of \cref{step:alg-interior-scalar-dual} of \cref{alg:alg-interior-scalar}]
	\label{rem:alg-interior-scalar-dual}
	The system on \cref{step:alg-interior-scalar-dual} is a standard \eqref{eq:ip-system-reg}. In the second-order cone with $A=\iprod{e}{\freevar}$ and $\iprod{e}{\range{K}}=\{0\}$, it is easy to solve. Indeed, $(0, \nexxt{\basepart d})=-K\thisx$ while $\nexxt{d_0}$ is given by the expression in \eqref{eq:d0-solution}. Finally
	\[
		\nexty
		=\mu_{i+1}\inv{(\nextd)}
		=\frac{\mu_{i+1}R\nextd}{\det(\nextd)}
		=\frac{\mu_{i+1}R\nextd}{(\nextd_0)^2-\norm{\nexxt{\basepart d}}^2}.
	\]
	More general cases $A=\iprod{a}{\freevar}$ and $\iprod{\inv a}{\range{K}}=\{0\}$ follow by scaling.
\end{remark}

We leave the solution of more general problems than the easy one considered in \cref{rem:alg-interior-scalar-dual} for future research. In particular, we would expect to combine the overall algorithm with a path-following interior point method in order to not have to solve the sub-problem exactly in each step, but to merely take a single step of the path-following method towards its solution.
Such an approach may yield a primal--dual version of the work in \cite{trandinh2014inexact}.

%%%
\subsection{Convergence rates in general symmetric cones}
%%%

We still need to specify $\mu_{i+1}$, verify \eqref{eq:interior0-sc-cond},  and produce convergence rates. In general symmetric cones, we have:

\begin{theorem}
    \label{thm:interior-scalar}
    With $\K$ an arbitrary symmetric cone, and $Z=\R^k$, let the requirements of \cref{alg:alg-interior-scalar} be satisfied.
    Assuming that $Ay=b$ implies $\iprod{a}{y}=b_0$ for some $a \in \iK$ and $b_0>0$, suppose there exists a solution $(\realoptx, \realopty, \realoptd, \realopt{z}) \in X \times \K \times \K \times Z$ to \eqref{eq:interior-oc} with $\realopty$ and $\realoptd$ strictly complementary, $\realoptd$ dual non-degenerate, and $\realopty$ primal non-degenerate.
    Suppose further that $\Dom G$ is bounded, or that the primal iterates $\{\thisx\}_{i \in \N}$ of \cref{alg:alg-interior-scalar} stay bounded through other means. 
    For some constant $\theta>0$ and $\zeta \in (0, b_0^{-2})$, take
    \begin{equation}
    	\label{eq:interior-scalar-choices}
        \mu_{i+1} \defeq \theta \tauTest_i^{-1/2},
        \quad\text{and}\quad
        \omegalowerbound{i+1} \defeq \zeta\eigval_{\min}(a)\mu_{i+1}.
    \end{equation}
    Then $\norm{x^N-\realoptx}^2=O(1/N)$.
\end{theorem}

\begin{remark}
    The assumption $Z=\R^k$ is merely for the simplicity of application of \cref{proposition:barrier-strong-monotonicity-sclp} and later \cref{corollary:barrier-strong-monotonicity-soc-extension}.
    There would be nothing stopping us from applying the results on uncountable products of symmetric cones, for example.
\end{remark}

\begin{proof}
    We use \cref{proposition:barrier-strong-monotonicity-sclp}, which verifies \eqref{eq:interior0-sc-cond} with 
    \[
    	\delta_{i+1} \le 
    	\hat C C_{-K\thisx,\mu_{i+1}} C_{-K\realoptx, \mu_{i+1}} \mu­_{i+1}
    	\quad\text{and}\quad
    	\omega_{i+1}=\omegalowerbound{i+1}=\zeta \eigval_{\min}(a)\mu_{i+1}
    \]
    for  $C_{-K\thisx,\mu_{i+1}}$, $C_{-K\realoptx, \mu_{i+1}}$ defined in \eqref{eq:ccmu}, and some $\hat C>0$.
    From \eqref{eq:ccmu} we see that the former constants are bounded as long as $\{\mu_i\}_{i \in \N}$ is non-increasing, and the sequence $\{\norm{K­\thisx}\}_{i \in \N}$ bounded. The latter is guaranteed by our assumptions, and the former by our construction of $\mu_{i+1}$ in \eqref{eq:interior-scalar-choices} and \cref{step:alg-interior-scalar-tautest} of the algorithm.
    Therefore $\delta_{i+1} \le C \mu_{i+1}$ for some constant $C>0$.
    From \eqref{eq:interior0-offdiag-cond} and \eqref{eq:interior-scalar-choices} it now follows
    \begin{equation}
    	\label{eq:interior-scalar-delta-est}
    	\Penalty_{i+1} \defeq \sigmaTest_{i+1}\sigma_{i+1}\delta_{i+1}
    	\le C \tau_i\tauTest_i\mu_{i+1}
    	= C\theta \tau_{i}\tauTest_{i}^{1/2}.
    \end{equation}

    Next we use \cref{thm:convergence-result-main} and \cref{lemma:interior0}.
    For $C_0 \defeq \frac{1}{2}\norm{u^0-\realoptu}^2_{\Test_{1}\Precond_{1}}$, \eqref{eq:convergence-result-main}, \eqref{eq:interior0-zimi-estim}, and \eqref{eq:interior-scalar-delta-est} give the combined estimate
    \begin{equation}
        \label{eq:interior-scalar-estim0}
        \frac{\tauTest_N}{2}\norm{x^N-\realoptx}^2
        \le
        C_0
        +
        C\theta \sum_{i=0}^{N-1} \tau_{i}\tauTest_{i}^{1/2},
        \quad
        (N \ge 1).
    \end{equation}
    Inserting $\omegalowerbound{i+1}$ and $\mu_{i+1}$ from \eqref{eq:interior-scalar-choices}, \cref{step:alg-interior-scalar-tautest,step:alg-interior-scalar-tau} of the algorithm say
    \[
    	\tauTest_{i+1} = \tauTest_i+\gamma\nu\tauTest_i^{1/2}
    	\quad\text{and}\quad
    	\tau_i=\tauTest_i^{-1/2}\nu/\norm{K}^2
    	\quad\text{for}\quad
    	\nu \defeq 2\zeta\lambda_{\min}(a)\theta.
    \]
    It follows \cite[see ][]{tuomov-cpaccel} that $\tauTest_N=\Theta(N^2)$, while $\sum_{i=0}^{N-1} \tau_{i}\tauTest_{i}^{1/2}=N\nu/\norm{K}^2$.
    Inserting these estimates into \eqref{eq:interior-scalar-estim0}, we verify the $O(1/N)$ rate.
\end{proof}

%%%
\subsection{Convergence rates in the second-order cone}
\label{sec:socalg}
%%%

In the second-order cone, we obtain linear convergence under dual non-degeneracy, $K\realoptx=0$.
In image processing example such as those we consider in \cref{sec:numeric}, we would have $Kx=(0, \grad x)$, lifting a discretised gradient to the second-order cone (or a pointwise product cone). Therefore $K\realoptx=0$ means that the solution image cannot be flat.

\begin{theorem}
    \label{thm:interior-scalar-soc}
    For $\K=\K_{\soc}$ the second-order cone, $Z=\R^k$, and $A=\iprod{a}{\freevar}$ for some $a \in \iK$ with $\iprod{\inv a}{\range{K}}=\{0\}$, let the requirements of \cref{alg:alg-interior-scalar} be satisfied.
    Suppose there exists a solution $(\realoptx, \realopty, \realoptd, \realopt{z}) \in X \times \K \times \K \times Z$ to \eqref{eq:interior-oc}.
    If $K\realoptx=0$, take $\realopty=b\inv a/2$ and $\realoptd=0$.
    For some $\theta>0$ and $\zeta \in (0, 2 b_0^{-2}]$, take
    \begin{equation}
    	\label{eq:interior-scalar-soc-choices}
        \mu_{i+1} \defeq  \theta \tauTest_i^{-1/2},
        \quad\text{and}\quad
        \omegalowerbound{i+1} \defeq (\mu_{i+1}\zeta + 2^{-1/2}\inv b_0 \norm{K\thisx}_{\inv Q_a})\eigval_{\min}(a).
    \end{equation}
    Suppose further that $\Dom G$ is bounded, or that the primal iterates $\{\thisx\}_{i \in \N}$ of \cref{alg:alg-interior-scalar} stay bounded through other means.
    Then for some $C,\varepsilon>0$ holds
    \[
        \norm{x^N-\realoptx}^2
        \le
        \begin{cases}
            C(1+\varepsilon)^{-N}, & K\realoptx \ne 0, \\
            C/N^2, & K\realoptx = 0.
        \end{cases} 
    \]
\end{theorem}

\begin{proof}
    From \cref{step:alg-interior-scalar-tau} of the algorithm and \eqref{eq:interior-scalar-soc-choices}, we expand
    \begin{equation}
        \label{eq:tau-expansion-soc}
        \tau_i \defeq 2(\zeta\theta \tauTest_i^{-1/2} + \tilde\ell_{i+1})\eigval_{\min}(a)/\norm{K}^2
        \quad\text{for}\quad
        \tilde\ell_{i+1} \defeq 2^{-1/2}\inv b_0 \norm{K\thisx}_{\inv Q_a}.
    \end{equation}
    From \eqref{eq:tau-expansion-soc} and \cref{step:alg-interior-scalar-tautest}, we estimate
    \begin{equation}
        \label{eq:tautest-soc-first-estim}
        \tauTest_N
        \ge \tauTest_0 + 2\gamma \zeta \theta \sum_{i=0}^{N-1} \tauTest_i^{1/2}.
    \end{equation}
    It follows from \eqref{eq:tau-expansion-soc} that $\sup_i \tau_i \le C_\tau$ for some constant $C_\tau>0$. From \eqref{eq:interior-scalar-soc-choices}, we also obtain $\mu_{i+1} \downto 0$.

    We then use  \cref{corollary:barrier-strong-monotonicity-soc-extension}, which verifies \eqref{eq:interior0-sc-cond} with
    \[
    	\omega_{i+1} \defeq (\mu_{i+1}\zeta + \ell_{i+1})\eigval_{\min}(a),
    	\quad
        \begin{cases}
        \ell_{i+1} \defeq \frac{\norm{K\thisx}_{\inv Q_a}}{b_0/\sqrt{2}},
        \ \text{and}\ %
        \PenaltyX_{i+1} \defeq 0, & \text{if } K\realoptx=0, 
        \\
        \ell_{i+1} = \frac{\norm{K\realoptx}_{\inv Q_a}+\norm{K\thisx}_{\inv Q_a}}{b_0/\sqrt{2}}, \ \text{and}\ %
        \PenaltyX_{i+1} = \mu_{i+1},
        & \text{if } K\realoptx \ne 0.
        \end{cases}
    \]
    Setting $\ell \defeq \sqrt{2}\norm{K\realoptx}_{\inv Q_a}/b_0 > 0$, we have $\ell_{i+1} = \tilde\ell_{i+1} + \ell$.

    Next we use \cref{thm:convergence-result-main} and \cref{lemma:interior0}.
    Recalling \eqref{eq:interior0-offdiag-cond} and that $\Penalty_{i+1}=\sigmaTest_{i+1}\sigma_{i+1}\PenaltyX_{i+1}$ in \cref{lemma:interior0}, setting $C_0 \defeq \frac{1}{2}\norm{u^0-\realoptu}^2_{\Test_{1}\Precond_{1}}$, \eqref{eq:convergence-result-main} and \eqref{eq:interior0-zimi-estim} yield
    \begin{equation}
        \label{eq:interior-scalar-estim0-soc}
        \frac{\tauTest_N}{2}\norm{x^N-\realoptx}^2
        \le
        C_0
        +
        D_N
        \quad
        \text{for}
        \quad
        D_N
        \defeq \sum_{i=0}^{N-1} \tau_i\tauTest_i\PenaltyX_{i+1}
        \quad        
        (N \ge 1).
    \end{equation}

    In the case $K\realoptx=0$, we have $\PenaltyX_{i+1}=0$. 
    As in the proof of \cref{thm:interior-scalar}, by a standard analysis \cite{tuomov-blockcp,tuomov-cpaccel}, it follows from \eqref{eq:tautest-soc-first-estim} that $\tauTest_N \ge CN^2$ for some $C>0$. We therefore get from \eqref{eq:interior-scalar-estim0-soc} the claimed $O(1/N^2)$ rate.

    Consider then the case $K\realoptx \ne 0$.
   	We estimate
    \begin{equation}
        \label{eq:soc-dn-first-estim}
        D_N
        = \sum_{i=0}^{N-1} \tau_i\tauTest_i\mu_{i+1}
        \le
        C_\tau \sum_{i=0}^{N-1} \tauTest_i\mu_{i+1}       
    \end{equation}
    By \cref{step:alg-interior-scalar-tau,step:alg-interior-scalar-tautest} of the algorithm, $\tauTest_N \ge \tauTest_0 + 2\gamma\zeta\norm{K}^{-2} \sum_{i=0}^{N-1} \tauTest_i\mu_{i+1}$.
    Using these estimates in \eqref{eq:interior-scalar-estim0-soc}, it follows that $\norm{x^N-\realoptx}$ is bounded. 
    If $\tilde\ell_{i+1} \downto 0$, \eqref{eq:tau-expansion-soc} and \eqref{eq:tautest-soc-first-estim} shows that also $\tau_i \downto 0$.
    Restarting our analysis from a later iteration, we can therefore make $C_\tau>0$ arbitrarily small. Consequently, for any $\epsilon>0$, for large enough $N$ holds $\norm{x^N-\realoptx} \le \epsilon$. Since $\ell>0$, this is in contradiction to $\tilde\ell_{i+1} \downto 0$. We may therefore assume that $\tilde\ell_{i+1} \ge \tilde\epsilon$ for some $\tilde\epsilon>0$, at least for large $i$. 
    Since our claims are asymptotical, we may without loss of generality assume this for all $i$.

	From \eqref{eq:tau-expansion-soc}, we now estimate $\tau_i \ge \tilde\epsilon\eigval_{\min}(a)/\norm{K}^2 =:\tau_*>0$. From \cref{step:alg-interior-scalar-tautest} consequently
    \begin{equation}
        \label{eq:soc-tautest-exponential}
        \tauTest_{i+1} \ge \tauTest_i(1+2\gamma\tau_*).
    \end{equation}
    This shows that $\tauTest_N \ge \Theta((1+\gamma­\tau_*)^N)$ grows exponentially, predicting \eqref{eq:interior-scalar-estim0-soc} to yield linear rates if we can control the penalty $D_N$.

    Continuing form \eqref{eq:soc-dn-first-estim}, by Hölder's inequality, since the conjugate exponent of $1/(1-p)$ is $1/p$, for any $p \in (0, 1)$ holds
    \[
        D_N
        \le C_\tau \theta \sum_{i=0}^{N-1} \tauTest_i^{1-p} \tauTest_i^{p-1/2}
        \le C_\tau \theta \left(\sum_{i=0}^{N-1} \tauTest_i\right)^{1-p} \left(\sum_{i=0}^{N-1} \tauTest_i^{1-1/(2p)}\right)^{p}.
    \]
    By \eqref{eq:soc-tautest-exponential}, the second sum on the right is bounded if $1-1/(2p)<0$, that is $p \in (0, 1/2)$.
    From \cref{step:alg-interior-scalar-tautest} of the algorithm
    \[
        \tauTest_N - \tauTest_0
        = 2\gamma \sum_{i=0}^{N-1} \tauTest_i \tau_i
        \ge 2\gamma \tau_* \sum_{i=0}^{N-1} \tauTest_i.
    \]
    For some constant $C'>0$ we therefore get
    \[
        D_N \le C' (\tauTest_N-\tauTest_0)^{1-p} \le C' \tauTest_N^{1-p}.
    \]
    Minding \eqref{eq:interior-scalar-estim0-soc} and \eqref{eq:soc-tautest-exponential}, this shows the claimed linear rate.
\end{proof}

%%%
\section{Numerical demonstrations}
\label{sec:numeric}
%%%

We study the performance of the proposed algorithm on two image processing problems, total variation (TV) denoising, and $H^1$ denoising. These can be written as
\begin{equation}
    \label{eq:denoising}
    \min_{x \in \R^{n_1n_2}}~ \frac{1}{2}\norm{z-x}_2^2 + \alpha R(x),
\end{equation}
where $n_1 \times n_2$ is the image size in pixels, and $z$ the noisy image as a vector in $\R^{n_1n_2}$. The parameter $\alpha>0$ is a regularisation parameter, and $R$ a regularisation term. For TV regularisation, it is $R(x)=\norm{D x}_{2,1}$, and for $H^1$ regularisation, it is $R(x)=\norm{D x}_2$. Here $D \in \R^{2n_1n_2 \times n_1n_2}$ is a matrix for a discretisation of the gradient, and $\norm{g}_{2,1} \defeq \sum_{i=1}^{n_1n_2} \sqrt{g_{i,1}+g_{i,2}}$ for $g=(g_{\cdot,1},g_{\cdot,2}) \in \R^{2n_1n_2}$.
We specifically take $D$ as forward-differences with Neumann boundary conditions.

The problem \eqref{eq:denoising} can in both cases be written in the saddle point form
\[
    \min_{x \in \R^{n_1n_2}} \max_{y \in \J}~
        \frac{1}{2}\norm{z-x}_2^2
        + \iprod{Kx}{y} - \delta_{\K \isect \inv A b}(y),
\]
where for $H^1$ denoising 
\begin{align*}
    \J & = \E_{1+2n_1n_2}, & Kx &=(0, Dx), & Ay & =y_0, & b &=\alpha,
    \\
\intertext{and for TV denoising, for $i=1,\ldots,n_1n_2$,}
    \J &=(\E_{1+2})^{n_1n_2}, & [Kx]_i &= (0, [Dx]_{i,1}, [Dx]_{i,2}), & 
    Ay&=((y_1)_0, \ldots (y_{n_1n_2})_0), & b &=(\alpha,\ldots,\alpha).
\end{align*}
In the latter case, \cref{step:alg-interior-scalar-dual} of \cref{alg:alg-interior-scalar} splits into $n_1n_2$ parallel problems of the form covered by \cref{rem:alg-interior-scalar-dual}. The remark therefore shows how to efficiently solve the step for both example problems.

While $\TV$ denoising \cite{Rud1992} is a fundamental benchmark in mathematical image processing, we have to emphasise here that $H^1$ denoising is not an approach of practical importance. It blurs images unlike $\TV$ denoising. Nevertheless, it forms a non-trivial optimisation problem, as we do not square the norm of the gradient. (The optimality conditions in that case would be linear: in the continuous setting the heat equation.)

\subsection{Remarks on convergence rates}

The linear convergence results for the second-order cone in \cref{sec:socalg} apply to $H^1$ denoising, but they do not apply to $\TV$ denoising. In the latter case, $\K=\K_{\soc}^{n_1n_2}$ is a product of second-order cones, but not a second-order cone.
It would be possible to extend the analysis of \cref{sec:socalg} to product cones. Due to the coupling through \eqref{eq:interior0-offdiag-cond}, a straightforward approach would yield linear convergence when $\min_i \norm{[K\realoptx]_i}>0$.
From the structure of the TV denoising problem, it is however easy to see that it can often happen that $[K\realoptx]_i=0$. This is the case when the solution image is locally flat. This happens in total variation denoising more often than one might expect, due to the characteristic staircasing effect of the approach \cite{ring2000structural}.
Therefore, there is little hope to obtain linear convergence on practical TV denoising problems using this approach.

\subsection{Numerical setup}

We performed some numerical experiments on the parrot image (\#23)  from the free Kodak image suite photo.\footnote{At the time of writing online at \url{http://r0k.us/graphics/kodak/}.}  We used the image, converted to greyscale, both at the original resolution of $n_1 \times n_2 = 768 \times 512$, and scaled down to $n_1 \times n_2 = 192 \times 128$ pixels.
Together with the dual variable, the problem dimensions are therefore $768 \cdot 512 \cdot 3 = 1179648 \simeq 10^6$ and $128 \cdot 128 \cdot 3 = 49152 \approx 4 \cdot 10^4$.
To the high-resolution test image, we added Gaussian noise with standard deviation $29.6$ ($12$dB). In the downscaled image, this becomes $6.15$ ($25.7$dB).
With the low-resolution image, we used regularisation parameter $\alpha=0.01$ for TV denoising, and $\alpha=5$ for $H^1$ denoising. We scale these up to $\alpha/0.25$ for the high-resolution image \cite{tuomov-tgvlearn}.

We compared our algorithm (denoted PEDI, \emph{Primal Euclidean--Dual Interior}) to the accelerated Chambolle--Pock method (PDHGM, \term{Primal--Dual Hybrid Gradient method, Modified} \cite{esser2010general}) on the saddle-point problem, as well as forward--backward splitting on the dual problem (Dual FB).
% and FISTA \cite{beck2009fista} on the dual problem.
For Dual FB we took as the basic step size $\tau=1/L^2$, where $L \defeq \sqrt{8} \ge \norm{K}$ \cite{chambolle2004meanalgorithm}. 
For the PDHGM, we took $\tau_0 \approx 0.52/L$ and $\sigma_0=1.9/L$, using the strong convexity parameter $\gamma=0.9<1$ for acceleration. For our method, we took $\zeta=0.9/b_0^2$ and $\theta=1/\zeta$, keeping $\tau_0$ and $\gamma$ unchanged from the PDHGM.
For the initial iterates we always took $x^0=0$ and $y^0=0$.
The hardware we used was a MacBook Pro with 16GB RAM and a 2.8 GHz Intel Core i5 CPU. The codes were written in MATLAB+C-MEX.

For our reporting, we computed a target optimal solution $\realoptx$ by taking one million iterations of the basic PDHGM.
In \cref{fig:tv-denoising,table:tv-denoising} for TV denoising, and \cref{fig:h1-denoising,table:h1-denoising} for $H^1$ denoising, we report the following: the distance to $\realoptx$ in decibels $10\log_{10}(\norm{x^i-\realoptx}^2/\norm{\realoptx}^2)$, the primal objective value $\text{val}(x) \defeq G(x)+F(Kx)$ relative to the target $10\log_{10}((\text{val}(x)-\text{val}(\hat x))^2/\text{val}(\hat x)^2)$, as well as the duality gap $10\log_{10}(\text{gap}^2/\text{gap}_0^2)$, again in decibels relative to the initial iterate.
For forward--backward splitting, to compute the duality gap, we solve the primal variable $\thisx$ from the primal optimality condition $K^*\thisy = \grad G(\thisx)=\thisx-z$.

\subsection{Performance analysis and concluding remarks}

As expected, the performance of PEDI on TV denoising is not particularly good, reflecting the $O(1/N)$ rates from \cref{thm:interior-scalar}.
For $H^1$ denoising we observe significantly improved convergence, reflecting the linear rates from \cref{thm:interior-scalar-soc}, and of dual forward--backward splitting.
While PEDI eventually has better gap behaviour than dual forward--backward splitting, overall, however, the method appears no match for the latter in our sample problems.
The results for the high resolution and low resolution problem are comparable. Since the low-resolution problem has size of order $10^4$, and the high resolution problem has size of the relatively large order $10^6$, this suggests good scalability of the algorithm.
Further research is required to see whether there are problems for which the overall Primal Euclidean(Proximal)--Dual Interior or similar approaches provide competitive algorithms. 

Irrespective of the limited practicality of PEDI, our theoretical analysis helps to bridge the gap in performance between direct primal or dual methods, and primal--dual methods. 
After all, we have obtained linear rates without the strong convexity of both $G$ and $F^*$ in the saddle point problem \eqref{eq:saddle}.
As a next step to take from here, it will be interesting to see if convergence rates can be derived in our overall setup for the ``distance-like'' preconditioners from \cite{yu2013convergence,chen2010entropy,lopez2017construction,kaplan2007interior}. Moreover, we are puzzled by what, if anything, makes the second-order cone special?
Finally, numerically we have only considered problems of the form given in \cref{rem:alg-interior-scalar-dual}, where the interior point sub-problem can be solved exactly.
This is sufficient for most image processing and similar applications.
However, it would be interesting to know whether we can combine a path-following interior point algorithm for its solution into the overall proximal point method. 
Such an approach may yield a primal--dual version of the work in \cite{trandinh2014inexact}.

\def\zinit{}
\def\finit{}
\def\tplot#1{\setlength{\figureheight}{0.19\textwidth}\setlength{\figurewidth}{0.23\textwidth}\scriptsize\input{#1}}

\begin{figure}[tbp!]
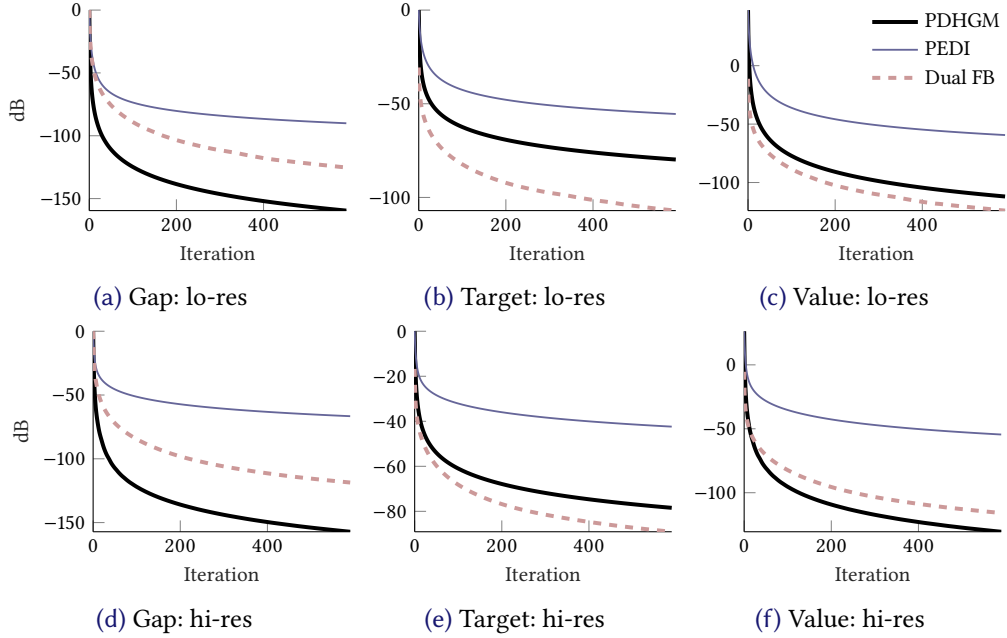

    \centering
    % zinit, lq
    \subcaptionbox{Gap: lo-res\zinit\label{fig:tv-denoising-deterministic-gap-lq-zinit}}{
        \setlength{\figureheight}{0.19\textwidth}\setlength{\figurewidth}{0.23\textwidth}\scriptsize\input{res/tv_denoising_results_lq_10000_gamma0.9_zinit_for_interiorcp_gaplog.tikz}%
    }
    \subcaptionbox{Target: lo-res\zinit\label{fig:tv-denoising-deterministic-target-lq-zinit}}{
        \setlength{\figureheight}{0.19\textwidth}\setlength{\figurewidth}{0.23\textwidth}\scriptsize\input{res/tv_denoising_results_lq_10000_gamma0.9_zinit_for_interiorcp_target2log.tikz}%
    }
    \subcaptionbox{Value: lo-res\zinit\label{fig:tv-denoising-deterministic-value-lq-zinit}}{
        \setlength{\figureheight}{0.19\textwidth}\setlength{\figurewidth}{0.23\textwidth}\scriptsize\input{res/tv_denoising_results_lq_10000_gamma0.9_zinit_for_interiorcp_dvallog.tikz}%
    }
    % zinit, hq
    \subcaptionbox{Gap: hi-res\zinit\label{fig:tv-denoising-deterministic-gap-hq-zinit}}{
        \setlength{\figureheight}{0.19\textwidth}\setlength{\figurewidth}{0.23\textwidth}\scriptsize\input{res/tv_denoising_results_hq_10000_gamma0.9_zinit_for_interiorcp_gaplog.tikz}%
    }
    \subcaptionbox{Target: hi-res\zinit\label{fig:tv-denoising-deterministic-target-hq-zinit}}{
        \setlength{\figureheight}{0.19\textwidth}\setlength{\figurewidth}{0.23\textwidth}\scriptsize\input{res/tv_denoising_results_hq_10000_gamma0.9_zinit_for_interiorcp_target2log.tikz}%
    }
    \subcaptionbox{Value: hi-res\zinit\label{fig:tv-denoising-deterministic-value-hq-zinit}}{
        \setlength{\figureheight}{0.19\textwidth}\setlength{\figurewidth}{0.23\textwidth}\scriptsize\input{res/tv_denoising_results_hq_10000_gamma0.9_zinit_for_interiorcp_dvallog.tikz}%
    }
    \caption{TV denoising convergence behaviour: high and low resolution images; gap, distance to target solution, and primal objective value in decibels.}
    \label{fig:tv-denoising}
\end{figure}

\begin{table}[tbp!]
    \caption{$\TV$ denoising performance: CPU time and number of iterations (at a resolution of 10) to reach given duality gap, distance to target, or primal objective value.}
    \label{table:tv-denoising}
    \centering
    \small
    \setlength{\tabcolsep}{2pt}
    \begin{tabular}{l|rr|rr|rr}
\multicolumn{7}{c}{low resolution}\\
\hline
 & \multicolumn{2}{c}{gap $\le -50$dB} & \multicolumn{2}{|c}{tgt $\le -50$dB} & \multicolumn{2}{|c}{val $\le -50$dB}\\
Method & iter & time & iter & time & iter & time\\
\hline
PDHGM & 4 & 0.01s & 30 & 0.09s & 27 & 0.08s\\
PEDI & 16 & 0.04s & 270 & 0.73s & 280 & 0.75s\\
Dual FB & 12 & 0.03s & 6 & 0.02s & 9 & 0.02s\\
\end{tabular}

    \,
    \begin{tabular}{rr|rr|rr}
\multicolumn{6}{c}{high resolution}\\
\hline
\multicolumn{2}{c}{gap $\le -50$dB} & \multicolumn{2}{|c}{tgt $\le -50$dB} & \multicolumn{2}{|c}{val $\le -50$dB}\\
iter & time & iter & time & iter & time\\
\hline
4 & 0.13s & 34 & 1.42s & 13 & 0.52s\\
86 & 3.78s & -- & -- & 400 & 17.76s\\
14 & 0.62s & 21 & 0.96s & 12 & 0.53s\\
\end{tabular}

\end{table}

%
% MATLAB2TIKZ fix: modify files with legend to set
% legend style={legend cell align=left, legend pos=north east, inner xsep=-3pt, inner ysep=-3pt, outer xsep=0pt, outer ysep=0pt, fill=none, draw=none}
%

\begin{figure}[tbp!]
    \centering
    % zinit, lq
    \subcaptionbox{Gap: lo-res\zinit\label{fig:h1-denoising-deterministic-gap-lq-zinit}}{
        \setlength{\figureheight}{0.19\textwidth}\setlength{\figurewidth}{0.23\textwidth}\scriptsize\input{res/h1_denoising_results_lq_10000_gamma0.9_zinit_for_interiorcp_gaplog.tikz}%
    }
    \subcaptionbox{Target: lo-res\zinit\label{fig:h1-denoising-deterministic-target-lq-zinit}}{
        \setlength{\figureheight}{0.19\textwidth}\setlength{\figurewidth}{0.23\textwidth}\scriptsize\input{res/h1_denoising_results_lq_10000_gamma0.9_zinit_for_interiorcp_target2log.tikz}%
    }
    \subcaptionbox{Value: lo-res\zinit\label{fig:h1-denoising-deterministic-value-lq-zinit}}{
        \setlength{\figureheight}{0.19\textwidth}\setlength{\figurewidth}{0.23\textwidth}\scriptsize\input{res/h1_denoising_results_lq_10000_gamma0.9_zinit_for_interiorcp_dvallog.tikz}%
    }
    % zinit, hq
    \subcaptionbox{Gap: hi-res\zinit\label{fig:h1-denoising-deterministic-gap-hq-zinit}}{
        \setlength{\figureheight}{0.19\textwidth}\setlength{\figurewidth}{0.23\textwidth}\scriptsize\input{res/h1_denoising_results_hq_10000_gamma0.9_zinit_for_interiorcp_gaplog.tikz}%
    }
    \subcaptionbox{Target: hi-res\zinit\label{fig:h1-denoising-deterministic-target-hq-zinit}}{
        \setlength{\figureheight}{0.19\textwidth}\setlength{\figurewidth}{0.23\textwidth}\scriptsize\input{res/h1_denoising_results_hq_10000_gamma0.9_zinit_for_interiorcp_target2log.tikz}%
    }
    \subcaptionbox{Value: hi-res\zinit\label{fig:h1-denoising-deterministic-value-hq-zinit}}{
        \setlength{\figureheight}{0.19\textwidth}\setlength{\figurewidth}{0.23\textwidth}\scriptsize\input{res/h1_denoising_results_hq_10000_gamma0.9_zinit_for_interiorcp_dvallog.tikz}%
    }
    \caption{$H^1$ denoising convergence behaviour: high and low resolution images; gap, distance to target solution, and primal objective value in decibels.}
    \label{fig:h1-denoising}
\end{figure}

\begin{table}[tbp!]
    \caption{$H^1$ denoising performance: CPU time and number of iterations (at a resolution of 10) to reach given duality gap, distance to target, or primal objective value.}
    \label{table:h1-denoising}
    \centering
    \small
    \setlength{\tabcolsep}{2pt}
    \begin{tabular}{l|rr|rr|rr}
\multicolumn{7}{c}{low resolution}\\
\hline
 & \multicolumn{2}{c}{gap $\le -150$dB} & \multicolumn{2}{|c}{tgt $\le -100$dB} & \multicolumn{2}{|c}{val $\le -100$dB}\\
Method & iter & time & iter & time & iter & time\\
\hline
PDHGM & 360 & 0.91s & -- & -- & 180 & 0.46s\\
PEDI & 120 & 0.31s & 87 & 0.22s & 54 & 0.14s\\
Dual FB & 44 & 0.11s & 43 & 0.11s & 22 & 0.05s\\
\end{tabular}

    \,%
    \begin{tabular}{rr|rr|rr}
\multicolumn{6}{c}{high resolution}\\
\hline
\multicolumn{2}{c}{gap $\le -150$dB} & \multicolumn{2}{|c}{tgt $\le -100$dB} & \multicolumn{2}{|c}{val $\le -100$dB}\\
iter & time & iter & time & iter & time\\
\hline
380 & 11.48s & -- & -- & 120 & 3.60s\\
51 & 1.69s & 39 & 1.28s & 24 & 0.78s\\
17 & 0.74s & 18 & 0.78s & 8 & 0.32s\\
\end{tabular}

\end{table}

\section*{Acknowledgements}

The final stages of this research have been performed with the support of the EPSRC First Grant EP/P021298/1, ``PARTIAL Analysis of Relations in Tasks of Inversion for Algorithmic Leverage''.

\section*{A data statement for the EPSRC}

The source code and data used to produce the numerical results of this publication have been deposited at \doi{10.5281/zenodo.1402031}.

\bibliography{bib/abbrevs,bib/bib-own,bib/bib}

\providecommand{\eprint}[1]{\href{http://arxiv.org/abs/#1}{arXiv:#1}}
  \providecommand{\eprint}[1]{\href{http://arxiv.org/abs/#1}{arXiv:#1}}
  \providecommand{\noopsort}[1]{}
\begin{thebibliography}{10}
\providecommand{\url}[1]{\texttt{#1}}
\providecommand{\urlprefix}{URL }
\expandafter\ifx\csname urlstyle\endcsname\relax
  \providecommand{\doi}[1]{doi:\discretionary{}{}{}#1}\else
  \providecommand{\doi}{doi:\discretionary{}{}{}\begingroup
  \urlstyle{rm}\Url}\fi
\providecommand{\eprint}[2][]{\url{#2}}

\bibitem{as-2003}
S.~H. Schmieta and F.~Alizadeh, \emph{Extension of primal-dual interior point
  algorithms to symmetric cones}, Mathematical Programming \textbf{96} (2003),
  409--438, \doi{10.1007/s10107-003-0380-z}.

\bibitem{faybusovich1997euclidean}
L.~Faybusovich, \emph{Euclidean jordan algebras and interior-point algorithms},
  Positivity \textbf{1} (1997), 331--357, \doi{10.1023/A:1009701824047}.

\bibitem{monteiro1998polynomial}
R.~Monteiro, \emph{{Polynomial Convergence of Primal-Dual Algorithms for
  Semidefinite Programming Based on the Monteiro and Zhang Family of
  Directions}}, SIAM Journal on Optimization \textbf{8} (1998), 797--812,
  \doi{10.1137/S1052623496308618}.

\bibitem{nesterov-todd}
Y.~E. Nesterov and M.~J. Todd, \emph{Self-scaled barriers and interior-point
  methods for convex programming}, Mathematics of Operations Research
  \textbf{22} (1997), 1--42, \doi{10.1137/S1052623495290209}.

\bibitem{cvx}
M.~Grant and S.~Boyd, \emph{{CVX}: Matlab software for disciplined convex
  programming, version 2.1}, \url{http://cvxr.com/cvx} (2014).

\bibitem{gb08}
M.~Grant and S.~Boyd, \emph{Graph implementations for nonsmooth convex
  programs}, in: \emph{Recent Advances in Learning and Control}, Edited by
  V.~Blondel, S.~Boyd and H.~Kimura, Lecture Notes in Control and Information
  Sciences, Springer-Verlag Limited2008, 95--110.

\bibitem{beck2009fista}
A.~Beck and M.~Teboulle, \emph{A fast iterative shrinkage-thresholding
  algorithm for linear inverse problems}, SIAM Journal on Imaging Sciences
  \textbf{2} (2009), 183--202, \doi{10.1137/080716542}.

\bibitem{loris2011generalization}
I.~Loris and C.~Verhoeven, \emph{On a generalization of the iterative
  soft-thresholding algorithm for the case of non-separable penalty}, Inverse
  Problems \textbf{27} (2011), 125007, \doi{10.1088/0266-5611/27/12/125007}.

\bibitem{gabay}
D.~Gabay, \emph{Applications of the method of multipliers to variational
  inequalities}, in: \emph{Augmented Lagrangian Methods: Applications to the
  Numerical Solution of Boundary-Value Problems}, volume~15, Edited by
  M.~Fortin and R.~Glowinski, North-Holland1983, 299--331.

\bibitem{chambolle2010first}
A.~Chambolle and T.~Pock, \emph{A first-order primal-dual algorithm for convex
  problems with applications to imaging}, Journal of Mathematical Imaging and
  Vision \textbf{40} (2011), 120--145, \doi{10.1007/s10851-010-0251-1}.

\bibitem{he2012convergence}
B.~He and X.~Yuan, \emph{Convergence analysis of primal-dual algorithms for a
  saddle-point problem: From contraction perspective}, SIAM Journal on Imaging
  Sciences \textbf{5} (2012), 119--149, \doi{10.1137/100814494}.

\bibitem{rockafellar1976monotone}
R.~T. Rockafellar, \emph{Monotone operators and the proximal point algorithm},
  SIAM Journal on Optimization \textbf{14} (1976), 877--898,
  \doi{10.1137/0314056}.

\bibitem{tuomov-proxtest}
T.~Valkonen, \emph{Testing and non-linear preconditioning of the proximal point
  method} (2017), submitted, \eprint{1703.05705}.
\newline\urlprefix\url{http://tuomov.iki.fi/m/proxtest.pdf}

\bibitem{tuomov-cpaccel}
T.~Valkonen and T.~Pock, \emph{Acceleration of the {PDHGM} on partially
  strongly convex functions}, Journal of Mathematical Imaging and Vision
  \textbf{59} (2017), 394--414, \doi{10.1007/s10851-016-0692-2},
  \eprint{1511.06566}.
\newline\urlprefix\url{http://tuomov.iki.fi/m/cpaccel.pdf}

\bibitem{tuomov-blockcp}
T.~Valkonen, \emph{Block-proximal methods with spatially adapted acceleration}
  (2017), submitted, \eprint{1609.07373}.
\newline\urlprefix\url{http://tuomov.iki.fi/m/blockcp.pdf}

\bibitem{yu2013convergence}
Z.~Yu, Y.~Zhu and Q.~Cao, \emph{On the convergence of central path and
  generalized proximal point method for symmetric cone linear programming},
  Appl. Math. Inf. Sci. \textbf{7} (2013), 2327--2333,
  \doi{10.12785/amis/070624}.

\bibitem{chen2010entropy}
J.-S. Chen and S.~Pan, \emph{An entropy-like proximal algorithm
  and the exponential multiplier method for convex symmetric cone
  programming}, Computational Optimization and Applications \textbf{47} (2010),
  477--499, \doi{10.1007/s10589-008-9227-0}.

\bibitem{lopez2017construction}
J.~López and E.~A.~P. Quiroz, \emph{Construction of proximal distances over
  symmetric cones}, Optimization  (2017), \doi{10.1080/02331934.2016.1277998},
  published online.

\bibitem{kaplan2007interior}
A.~Kaplan and R.~Tichatschke, \emph{Interior proximal method for variational
  inequalities on non-polyhedral sets}, Discussiones Mathematicae, Differential
  Inclusions, Control and Optimization \textbf{27} (2007), 71--93,
  \doi{10.7151/dmdico1077}.

\bibitem{trandinh2014inexact}
Q.~Tran-Dinh, A.~Kyrillidis and V.~Cevher, \emph{An inexact proximal
  path-following algorithm for constrained convex minimization}, SIAM Journal
  on Optimization \textbf{24} (2014), 1718--1745, \doi{10.1137/130944539}.

\bibitem{chen1997convergence}
G.~H.-G. Chen and R.~T. Rockafellar, \emph{Convergence rates in
  forward--backward splitting}, SIAM Journal on Optimization \textbf{7} (1997),
  421--444, \doi{10.1137/S1052623495290179}.

\bibitem{boley2013local}
D.~Boley, \emph{Local linear convergence of the alternating direction method of
  multipliers on quadratic or linear programs}, SIAM Journal on Optimization
  \textbf{23} (2013), 2183--2207, \doi{10.1137/120878951}.

\bibitem{han2013local}
D.~Han and X.~Yuan, \emph{Local linear convergence of the alternating direction
  method of multipliers for quadratic programs}, SIAM Journal on Numerical
  Analysis \textbf{51} (2013), 3446--3457, \doi{10.1137/120886753}.

\bibitem{hong2017linear}
M.~Hong and Z.-Q. Luo, \emph{On the linear convergence of the alternating
  direction method of multipliers}, Mathematical Programming \textbf{162}
  (2017), 165--199, \doi{10.1007/s10107-016-1034-2}.

\bibitem{liang2014local}
J.~Liang, J.~Fadili and G.~Peyr\'{e}, \emph{Local linear convergence of
  forward--backward under partial smoothness}, Advances in Neural Information
  Processing Systems \textbf{27} (2014), 1970--1978.
\newline\urlprefix\url{http://papers.nips.cc/paper/5260-local-linear-convergence-of-forward-backward-under-partial-smoothness.pdf}

\bibitem{bredies2008linear}
K.~Bredies and D.~A. Lorenz, \emph{Linear convergence of iterative
  soft-thresholding}, Journal of Fourier Analysis and Applications \textbf{14}
  (2008), 813--837, \doi{10.1007/s00041-008-9041-1}.

\bibitem{rockafellar-convex-analysis}
R.~T. Rockafellar, \emph{Convex Analysis}, Princeton University Press1972.

\bibitem{faraut-koranyi-symmetric}
J.~Faraut and A.~Kor{\'a}nyi, \emph{Analysis on Symmetric Cones}, Oxford
  University Press1994.

\bibitem{koecher-notes}
M.~Koecher, \emph{The {M}innesota notes on {J}ordan algebras and their
  applications}, \emph{Lecture Notes in Mathematics}, volume 1710,
  Springer-Verlag, Berlin1999.

\bibitem{faybusovich-1997}
L.~Faybusovich, \emph{Linear systems in {J}ordan algebras and primal-dual
  interior-point algorithms}, Journal of Computational and Applied Mathematics
  \textbf{86} (1997), 149--175, \doi{10.1016/S0377-0427(97)00153-2}.

\bibitem{tuomov-thesis}
T.~Valkonen, \emph{Diff-convex combinations of {E}uclidean distances: a search
  for optima}, number~99 in Jyv{\"a}skyl{\"a} Studies in Computing, University
  of Jyv{\"a}skyl{\"a}2008, {Ph.D} Thesis.
\newline\urlprefix\url{http://tuomov.iki.fi/m/thesis.pdf}

\bibitem{wright2002properties}
S.~J. Wright and D.~Orban, \emph{Properties of the log-barrier function on
  degenerate nonlinear programs}, Mathematics of Operations Research
  \textbf{27} (2002), 585--613, \doi{10.1287/moor.27.3.585.312}.

\bibitem{ramirez2016central}
H.~Ram{\'i}rez and D.~Sossa, \emph{On the central paths in symmetric cone
  programming}, Journal of Optimization Theory and Applications  (2016), 1--20,
  \doi{10.1007/s10957-016-0989-8}.

\bibitem{dacruzneto2008central}
J.~X. da~Cruz~Neto, O.~P. Ferreira, P.~R. Oliveira and R.~C.~M. Silva,
  \emph{Central paths in semidefinite programming, generalized proximal-point
  method and cauchy trajectories in riemannian manifolds}, Journal of
  Optimization Theory and Applications \textbf{139} (2008), 227,
  \doi{10.1007/s10957-008-9422-2}.

\bibitem{monteiro1998existence}
R.~D. Monteiro and F.~Zou, \emph{On the existence and convergence of the
  central path for convex programming and some duality results}, Computational
  Optimization and Applications \textbf{10} (1998), 51--77,
  \doi{10.1023/A:1018339901042}.

\bibitem{burachik2008properties}
R.~S. Burachik, L.~M.~G. Drummond and S.~Scheimberg, \emph{On some properties
  and an application of the logarithmic barrier method}, Mathematical
  Programming \textbf{111} (2008), 95--112, \doi{10.1007/s10107-006-0075-3}.

\bibitem{halicka2005limiting}
M.~Halická, E.~de~Klerk and C.~Roos, \emph{Limiting behavior of the central
  path in semidefinite optimization}, Optimization Methods and Software
  \textbf{20} (2005), 99--113, \doi{10.1080/10556780410001727718}.

\bibitem{alizadeh2003soc}
F.~Alizadeh and D.~Goldfarb, \emph{Second-order cone programming}, Mathematical
  Programming \textbf{95} (2003), 3--51, \doi{10.1007/s10107-002-0339-5}.

\bibitem{Rud1992}
L.~Rudin, S.~Osher and E.~Fatemi, \emph{Nonlinear total variation based noise
  removal algorithms}, Physica D \textbf{60} (1992), 259--268.

\bibitem{ring2000structural}
W.~Ring, \emph{Structural properties of solutions to total variation
  regularization problems}, ESAIM: Mathematical Modelling and Numerical
  Analysis \textbf{34} (2000), 799--810, \doi{10.1051/m2an:2000104}.

\bibitem{tuomov-tgvlearn}
J.~C. de~Los~Reyes, C.-B. Sch{\"o}nlieb and T.~Valkonen, \emph{Bilevel
  parameter learning for higher-order total variation regularisation models},
  Journal of Mathematical Imaging and Vision \textbf{57} (2017), 1--25,
  \doi{10.1007/s10851-016-0662-8}, \eprint{1508.07243}.
\newline\urlprefix\url{http://tuomov.iki.fi/m/tgv_learn.pdf}

\bibitem{esser2010general}
E.~Esser, X.~Zhang and T.~F. Chan, \emph{A general framework for a class of
  first order primal-dual algorithms for convex optimization in imaging
  science}, SIAM Journal on Imaging Sciences \textbf{3} (2010), 1015--1046,
  \doi{10.1137/09076934X}.

\bibitem{chambolle2004meanalgorithm}
A.~Chambolle, \emph{An algorithm for mean curvature motion}, Interfaces and
  Free Boundaries \textbf{6} (2004), 195.

\end{thebibliography}

\end{document}